\documentclass[11pt,reqno]{amsart}
\usepackage{amssymb,latexsym}
\usepackage{graphicx}
\usepackage{url}		%does nice formatting of URLs
\usepackage{blindtext}
\usepackage{hyperref}
\usepackage{breakurl}
\newcommand{\bburl}[1]{\textcolor{blue}{\url{#1}}}
\usepackage{multirow}
\usepackage{multicol}
\usepackage{listings}
\usepackage{tikz}
\usepackage{subfig}
\usepackage{amssymb}
\usepackage{xcolor}
\usepackage{float}
\usepackage{amsmath,amsfonts}
\calclayout

\makeatletter
\newcommand{\monthyear}[1]{%
  \def\@monthyear{\uppercase{#1}}}
\newcommand{\volnumber}[1]{%
  \def\@volnumber{\uppercase{#1}}}
\AtBeginDocument{%
\def\ps@plain{\ps@empty
  \def\@oddfoot{\@monthyear \hfil \thepage}%
  \def\@evenfoot{\thepage \hfil \@volnumber}}
\def\ps@firstpage{\ps@plain}
\def\ps@headings{\ps@empty
  \def\@evenhead{%
    \setTrue{runhead}%
    \def\thanks{\protect\thanks@warning}%
    \uppercase{The Fibonacci Quarterly}\hfil}%
  \def\@oddhead{%
    \setTrue{runhead}%
    \def\thanks{\protect\thanks@warning}%
    \hfill\uppercase{SUMS OF RECIPROCALS OF RECURRENCE RELATIONS}}%
  \let\@mkboth\markboth
  \def\@evenfoot{%
    \thepage \hfil \@volnumber}%
  \def\@oddfoot{%
    \@monthyear \hfil \thepage}%
  }%
\footskip=25pt
\pagestyle{headings}%
}
\makeatother

\theoremstyle{plain}
\numberwithin{equation}{section}
\newtheorem{thm}{Theorem}[section]

\newtheorem{lem}[thm]{Lemma}

\newtheorem{defi}[thm]{Definition}

\newtheorem{rek}[thm]{Remark}
\newtheorem{cor}[thm]{Corollary}
\newtheorem{conj}[thm]{Conjecture}

\thanks{We thank the referee, the participants of the 20th International Fibonacci Conference, and Professor Evan O'Dorney for many helpful comments on an earlier version of this paper.}

\begin{document}
%% replace the values in the next three lines by the correct information
\monthyear{December 2022}
\volnumber{Volume, Number}
%%%%%%%%%%%%%what should we put here?%%%%%%%%%%%%%%%%%%%%%%%%%%%%%%%%%%%%%%%%%%%%%%%%%%%%%%%%%%%%%%%%%%%%%%%%%%%%%%%%%%%%%%%%%%%%%
\setcounter{page}{1}

\title{Sums of Reciprocals of Recurrence Relations}
\author{Hao Cui}
\address{Marc Garneau Collegiate Institute\\
                Toronto, Ontario\\
                M3C 1B3, Canada}
\email{3031556720qq@gmail.com}
\author{Xiaoyu Cui}
\address{Princeton High School\\
               Princeton, New Jersey\\
               08540, USA}
\email{raincui1020@gmail.com}
\author{Sophia C. Davis}
\address{Department of Astronomy\\
               University of Michigan\\
               Ann Arbor, Michigan\\
               48109, USA}
\email{sophiacd@umich.edu}
\author{Irfan Durmi\'{c}}
\address{Department of Mathematics and Statistics\\
                University of Jyväskylä\\
               Jyväskylä\\
               40740, FI}
\email{idurmic@student.jyu.fi}
\author{Qingcheng Hu}
\address{Shanghai Starriver Bilingual School\\
                Shanghai\\
                201108, China}
\email{jackqchu@outlook.com}
\author{Lisa Liu}
\address{Concord Academy\\
               Concord, Massachusetts\\
                01742, USA}
\email{lisachangliu@outlook.com}
\author{Steven J. Miller}
\address{Department of Mathematics and Statistics\\
                Williams College\\
                Williamstown, Massachusetts\\
                01267, USA}
\email{sjm1@williams.edu}
\email{Steven.Miller.MC.96@aya.yale.edu}
\author{Fengping Ren}
\address{Taft school \\
                Watertown, Connecticut\\
                06795, USA}
\email{fengpingren@gmail.com}
\author{Alicia Smith Reina}
\address{Mathematical Institute\\
                University of Oxford\\
                Oxford\\
                United Kingdom}
\email{alicita.smith01@gmail.com}
\author{Eliel Sosis}
\address{Department of Mathematics\\
               University of Michigan\\
               Ann Arbor, Michigan\\
               48109, USA}
\email{esosis@umich.edu}

%\noindent MSC2020: 11B39, 33C05



%%%%%%%%%%%%%%%%%%%%%%%%%%%%%%%%%%%%%%%%%%%%%%%%%%%%%%%%%%%%%%%%%%%%%%%%%%%%%%%%%%%%%%%%%%%%%%%%%%%%%%%%%%%%%%%%%%%%%%%%%%%%%%%%%%%%%%%%%%%%%%%%%%%%%%%%%%%%%%%%%%%%%%%%%%%%%%%%%%%%%%%%%%%%%%%%%%%%%%%%%%%%%%%%%%%%%%%%%%%%%%%%%%%%%%%%%%%%%%%%%%%%%%%%%%%%%%%%%%%%%%%%%%%%%%%%%%%%%%%%%%%%%%%%%%%%%%%%%%%%%%%%%%%%%%%%%%%%%%%%%%%%%%%%%%%%%%%%%%%%%%%%%%%%%%%%%%%%%%%%%%%%%%%%%%%%%%%%%%%%%%%%%%%%%%%%%%%%%%%%

\begin{abstract}
There is a growing literature on sums of reciprocals of polynomial functions of recurrence relations with constant coefficients and fixed depth, such as Fibonacci and Tribonacci numbers, products of such numbers, and balancing numbers (numbers $n$ such that the sum of the integers less than $n$ equals the sum of the $r$ integers immediately after, for some $r$ which is called the balancer of $n$; If $n$ is included in the summation, we have the cobalancing numbers, and $r$ is called the cobalancer of $n$). We generalize previous work to reciprocal sums of depth two recurrence sequences with arbitrary coefficients and the Tribonacci numbers, and show our method provides an alternative proof of some existing results. 

We define $(a,b)$ balancing and cobalancing numbers, where $a$ and $b$ are constants that multiply the left-hand side and right-hand side respectively, and derive recurrence relations describing these sequences. We show that for balancing numbers, the coefficients $(3,1)$ is unique such that every integer is a $(3,1)$ balancing number, and proved there does not exist an analogous set of coefficients for cobalancing numbers. We also found patterns for certain coefficients that have no balancing or cobalancing numbers.
%\vspace{-3pc}
\end{abstract}

\maketitle

%%%%%%%%%%%%%%%%%%%%%%%%%%%%%%%%%%%%%%%%%%%%%%%%%%%%%%%%%%%%%%%%%%%%%%%%%%%%%%%%%%%%%%%%%%%%%%%%%%%%%%%%%%%%%%%%%%%%%%%%%%%%%%%%%%%%%%%%%%%%%%%%%%%%%%%%%%%%%%%%%%%%%%%%%%%%%%%%%%%%%%%%%%%%%%%%%%%%%%%%%%%%%%%%%%%%%%%%%%%%%%%%%%%%%%%%%%%%%%%%%%%%%%%%%%%%%%%%%%%%%%%%%%%%%%%%%%%%%%%%%%%%%%%%%%%%%%%%%%%%%%%%%%%%%%%%%%%%%%%%%%%%%%%%%%%%%%%%%%%%%%%%%%%%%%%%%%%%%%%%%%%%%%%%%%%%%
\section{Introduction}
The Fibonacci numbers have numerous interesting properties and applications; see for example \cite{Kos}. We take as their definition $F_0 = 0$, $F_{1} = 1$ and $F_{n} = F_{n-1}+F_{n-2}$. In addition to studying these, we examine several generalizations, including the Tribonacci numbers, defined as $T_{0} = 0$, $T_{1} = T_{2} = 1$  and $T_{n} =  T_{n-1}+T_{n-2}+T_{n-3}$.

Ohtsuka and Nakamura \cite{ON} derived the following formula for infinite reciprocal sums of consecutive Fibonacci numbers, where $\lfloor x \rfloor$ is the greatest integer at most $x$:
\begin{equation}\label{AI1}
    \left \lfloor\left(\sum_{k=n}^{\infty}\frac{1}{F_{k}}\right)^{-1} \right \rfloor\ =\
\begin{cases}
F_{n-2}  & \text{ if $n$ is even and $n$}\ \geq\ 2\\
F_{n-2}-1 & \text{ if $n$ is odd and $n$}\ \geq\ 1.\\
\end{cases}
\end{equation}

Anantakitpaisal and Kuhapatanakul \cite{AK} extended this result to the Tribonacci numbers,
\begin{equation}\label{AI2}
    \left \lfloor\left(\sum_{k=n}^{\infty}\frac{1}{T_{k}}\right)^{-1} \right \rfloor\ =\
\begin{cases}
    T_{n}-T_{n-1}  & \text{ if }T_{-\left(n+1\right)}\ <\ 0\\
    T_{n}-T_{n-1}-1 & \text{ if }T_{-\left(n+1\right)}\ >\ 0,\\
\end{cases}
\end{equation}
while Komatsu \cite{Ko} proved a formula for the nearest integer of such sums,
\begin{equation}\label{AI3}
    \left \lfloor\left(\sum_{k=n}^{\infty}\frac{1}{T_{k}}\right)^{-1}+\frac{1}{2} \right \rfloor\ = \ T_{n}-T_{n-1}.
\end{equation}
Given the results above, it is natural to ask if they hold for other recurrence relations. 
%The definition of balancing numbers involves an integer that is not involved in the two equal sums (the sum from integers $1$ to $k-1$ and the sum from integers $k+1$ to $k+r$). 
%\irfan{Also, be consistent with your use of the \$\$ environment: ideally, use it for all numbers and expressions. Related to Alicia's comment below, it might be useful for you to write an example in order to clarify what a balancing number is}
%\alicia{From this it is unclear as to what the definition of balancing numbers is, it would be better if you just write out the definition instead of mentioning what it involves. You could even write something similar to what you wrote in the abstract}
Balancing numbers are numbers $n$ such that the sum of the integers less than $n$ equal the sum of the $r$ integers immediately after, for some $r$ which is called the balancer of $n$. For example, $6$ is a balancing number with balancer $2$ because
\begin{equation}\nonumber
    1 + 2 + 3 + 4 + 5\ =\ 7 + 8.
\end{equation}
Behera and Panda \cite{BP} showed that balancing numbers follow the recurrence relation 
\begin{equation}\label{AI6}
    B_{n+1} \ =\  6B_{n} - B_{n-1},
\end{equation}
with initial terms $B_1 = 6$ and $B_2 = 35$, where $B_{n}$ is the $n$\textsuperscript{{\rm th}} balancing number. The reciprocal sum of balancing numbers has been proven by \cite{PKD}:
\begin{equation}\label{AI4}
\left\lfloor\left(\sum_{k = n}^{\infty} \frac{1}{B_{k}}\right)^{-1}\right\rfloor \ = \ B_n-B_{n-1}-1.
\end{equation}

If $n$ is included in the summation on the left-hand side we have the cobalancing numbers, first introduced by Panda in \cite{Pa}, and $r$ is called the cobalancer of $n$.
For the cobalancing numbers, Panda and Ray \cite{Pa,Ra} showed that
\begin{equation}\label{AI5}
    b_{n+1} \ =\  6b_{n} - b_{n-1} + 2,
\end{equation}
with initial terms $b_1 = 2$ and $b_2 = 14$, where $b_{n}$ is the $n$\textsuperscript{{\rm th}} cobalancing number.

\begin{rek} While the cobalancing numbers do not come from a depth two recurrence with constant coefficients, a trivial modification does. Consider the shifted sequence $c_n = b_n - d$, with $d$ to be determined. 
Then \begin{eqnarray} b_{n+1} - d & \ = \ & 6(b_n - d) - (b_{n-1} - d) + 2 \nonumber\\
b_{n+1} - d &=& 6 b_n - b_{n-1} + 2 - 5d \nonumber\\
b_{n+1} &=& 6 b_n - b_{n-1} + 2 - 4d, 
\end{eqnarray} 
where $b_{n}$ is the $n$\textsuperscript{{\rm th}} cobalancing number. Thus if we take $d = 1/2$ then $b_n = c_n + d$ satisfies a depth two constant coefficient recurrence relation.
\end{rek}

We generalize these definitions to define the $(a,b)$ balancing and cobalancing numbers, where $a$, $b$ are constants that multiply the left and right-hand sides.

\begin{defi}
Let $a, b \in \mathbb{Z}^+$ be coprime, the $(a,b)$ balancing numbers are positive integers $n$ such that the equality
\begin{equation}\label{GI1}
    a\left(1+2+\cdots+(n-1)\right)\ = \ b\left((n+1)+(n+2)+\cdots+(n+r)\right)
\end{equation}
is satisfied for some positive integer $r$, and $r$ is called the balancer of $n$.
\end{defi}

\begin{defi}
Let $a, b \in \mathbb{Z}^+$ be coprime, the $(a,b)$ cobalancing numbers are positive integers $n$ such that the equality
\begin{equation}\label{GI2}
    a\left(1+2+\cdots+(n-1)+n\right)\ = \ b\left((n+1)+(n+2)+\cdots+(n+r)\right)
\end{equation}
is satisfied for some positive integer $r$, and $r$ is called the cobalancer of $n$.
\end{defi}
%\alicia{Need to define what the $b_n$ represent, if they represent the $nth$ cobalancing number then you should make that explicit}

We find recurrences describing $(a,b)$ balancing and cobalancing numbers, including numerical solutions for cases where $a\leq7$, $b\leq 5$, and we derive analytically depth-two recurrence relations for all $(a,b)$ cobalancing numbers and their corresponding cobalancers with $a\in \{ 1,2\}$.

\begin{thm} \label{gen-cobalancing}
    All $(a,b)$ cobalancing numbers such that $a \in \{1,2\}$ can be described by a depth two recurrence plus a constant term of the form 
\begin{equation}\label{BI1}
    c_{n} \ = \ (2m+2) c_{n-1}- c_{n-2}+m,
\end{equation}
where $m = 2b/a$. For any of the depth two recurrences of this form, the sequence of $(a,b)$ cobalancing numbers starts with $c_{1}=m, c_{2}=2m^2+3m$. 
\end{thm}

\begin{thm} \label{gen-cobalancer}
The cobalancers of $(a,b)$ cobalancing numbers such that $a \in \{1,2\}$ can be described by a depth two recurrence of the form
\begin{equation}\label{BI2}
    r_{n} \ = \ (2m+2) r_{n-1}- r_{n-2},
\end{equation}
where $r_{n}$ is the cobalancer of the $n^{th}$ cobalancing number and $m=2b/a$. For any of the depth two recurrences of this form, the sequence of $(a,b)$ cobalancers starts with $r_{1}=1, r_{2}=2m+2$. 
\end{thm}

We then explore the infinite reciprocal sums of depth two recurrence sequences. Our two main results on these recurrences that apply to sequences describing cobalancing numbers are shown below. The proofs are similar to Theorem $2.1$ in \cite{PKD} which proves the result for reciprocal sums of balancing numbers. In Theorem \ref{thm:II0} we generalize the result to arbitrary coefficients. Completing the proof required certain restrictions on the coefficients, along with Lemma \ref{lem: II0} to find the square of the $n\textsuperscript{{\rm th}}$ term in the sequence. In the proof of Theorem \ref{thm:II1} we also consider arbitrary coefficients, and we see a new feature that the square of the $n^{th}$ term in the sequence depends on the parity of $n$. We handle it by separating the proof into two cases; one when $n$ is even, and another when $n$ is odd.

\begin{thm}\label{thm:II0}
For all recurrences of the form
\begin{equation}\label{II1}
c_{n+1} \ =  \ q c_{n}-c_{n-1}+s,
\end{equation}
where $q, s \in \mathbb{R}_{\ne 0}, q \geq 2, c_{0}=0, c_{1}=s$, if $s > \frac{1}{2}$ and
\begin{equation}\label{II2}
0 \ \leq \ (q-s) c_{n} - 2c_{n-1} +s-1,
\end{equation}
then for any positive integer $n$, we have
\begin{equation}\label{thm:II3}
\left\lfloor\left(\sum_{k = n}^{\infty} \frac{1}{c_{k}}\right)^{-1}\right\rfloor \ = \ c_n-c_{n-1}-1,
\end{equation}
where $c_{n}$ is the $n$\textsuperscript{{\rm th}} term in the sequence.
\end{thm}

\begin{thm}\label{thm:II1}
For all recurrences of the form
\begin{equation}\nonumber
c_{n+1} \ = \ q c_{n}+r c_{n-1},
\end{equation}
where $q, r \in \mathbb{R}_{\ne 0}, c_{0}=0$, $c_{1} = t$, and $t>0$, we have the following cases.\\

\emph{Case 1}: when $q\geq 3$ and  $-1\leq r<0$, if 
\begin{equation*}
    t^2(-r)^{n-1} \ \leq\  c_{n+1}-c_{n-1}-1 
\end{equation*}
then for any positive integer $n$,
\begin{equation}\nonumber
\left\lfloor\left(\sum_{k = n}^{\infty} \frac{1}{c_{k}}\right)^{-1}\right\rfloor \ = \ c_{n}-c_{n-1}-1.
\end{equation}

\emph{Case 2}: when $q\geq 2$ and $r\geq 0$, if 
\begin{equation}\nonumber
\begin{cases}
t^2(-r)^{n-1} \ \leq\  c_{n+1}-c_{n-1}-1 & \text{{\rm if $n$ is odd}}\\
t^2(-r)^{n-1} \ >\   -c_{n+1}+c_{n-1}-1  & \text{{\rm if $n$ is even}},\\
\end{cases}
\end{equation}
then for any positive integer $n$, we have
\begin{equation}\nonumber
\left\lfloor\left(\sum_{k = n}^{\infty} \frac{1}{c_{k}}\right)^{-1}\right\rfloor \ = \ 
\begin{cases}
c_{n}-c_{n-1}-1 & \text{{\rm if $n$ is odd and $n$}} \geq 1\\
c_{n}-c_{n-1}  & \text{{\rm if $n$ is even and $n$ }} \geq 2,\\
\end{cases}
\end{equation}
where $c_{n}$ is the $n$\textsuperscript{{\rm th}} term in the sequence.
\end{thm}

We also investigate further the floor and nearest integer of infinite reciprocal sums of every other Tribonacci number, and then generalize the results to the reciprocal sums of every $n$\textsuperscript{th} Tribonacci number. In addition, we explore the alternating sum of Tribonacci numbers and sum of generalized Tribonacci numbers. The main result we obtained on Tribonacci numbers is as follows.

\begin{thm}\label{sum of every nth}
Let $m$ be a positive integer. For large enough $n$, we have that
\begin{equation}\label{thm:DI2}
    \left\{\left(\sum^{\infty}_{k = 0} \frac{1}{T_{n + mk}}\right)^{-1}\right\}\ =\ T_{n} - T_{n-m},
\end{equation}
where $\{j\}$ denotes the closest integer to $j$.
\end{thm}

We then prove several interesting results for specific cases of $(a,b)$ balancing and cobalancing numbers.

\begin{thm}\label{thm:JI1}
If $a$ and $b$ are relatively prime integers, then $(3, 1)$ is the only choice of $(a,b)$ for which every positive integer $n$ is an $(a,b)$ balancing number.
\end{thm}
%\alicia{Your use of the term 'unique' is a bit confusing. Do you mean that $(a,b) = (3,1)$ are the only $(a,b)$ coefficients that make every integer $n$ a balancing number?}

\begin{thm} \label{thm:JI2}
Coefficients $(a,b)$ do not exist such that every positive integer $n$ is an $(a,b)$ cobalancing number.
\end{thm}

%\begin{thm} \label{thm:JI3}
%The only balancing number for coefficients $(8,1)$ is $n=1$ with a corresponding balancer $r=0$.
%\end{thm}
%\irfan{Theorem~\ref{thm:JI3} seems honestly more like an interesting corollary so I would rethink this}
\begin{thm} \label{thm:JI4}
For all coefficients $(a,b)$ such that $a=16y^2+16y+3$ and $b=1$, where $y$ is a positive integer, the only cobalancing number will be $n=y$ with corresponding cobalancer $r=4y^2+3y$.
\end{thm}

Finally, we define square balancing numbers as positive integers $n$ such that the sum of the squares of integers less than $n$ equal the sum of the squares of $r$ integers immediately after, for some positive integer $r$ which is called the square balancer of $n$. If $n$ is included in the summation, we have the square cobalancing numbers, and $r$ is called the square cobalancer of $n$. We then explore interesting patterns for $(a,b)$ square balancing and cobalancing numbers.
%\alicia{What are square balancing and cobalancing numbers? Add a very brief definition}

%%%%%%%%%%%%%%%%%%%%%%%%%%%%%%%%%%%%%%%%%%%%%%%%%%%%%%%%%%%%%%%%%%%%%%%%%%%%%%%%%%%%%%%%%%%%%%%%%%%%%%%%%%%%%%%%%%%%%%%%%%%%%%%%%%%%%%%%%%%%%%%%%%%%%%%%%%%%%%%%%%%%%%%%%%%%%%%%%%%%%%%%%%%%%%%%%%%%%%%%%%%%%%%%%%%%%%%%%%%%%%%%%%%%%%%%%%%%%%%%%%%%%%%%%%%%%%%%%%%%%%%%%%%%%%%%%%%%%%%%%%%%%%%%%%%%%%%%%%%%%%%%%%%%%%%%%%%%%%%%%%%%%%%%%%%%%%%%%%%%%%%%%%%%%%%%%%%%%%%%%%%%%%%%%%%%%%%%%%%%%%%%%%%%%%%%%%%%%%%%

\section{Recurrences for (a,b) Balancing and Cobalancing Numbers}
Using the code attached in Appendix \ref{code}, we found recurrence relations for many $(a,b)$ balancing and cobalancing numbers and their corresponding balancers and cobalancers. Interestingly, the $(1,1)$ balancers are equivalent to the $(1,1)$ cobalancing numbers, and the $(1,1)$ cobalancers are equivalent to the $(1,1)$ balancing numbers, except for the inclusion of $1$ as the first cobalancer, but not as the first balancing number. 

The tables below contain recurrences for $(a,b)$ balancing and cobalancing numbers and their corresponding balancers and cobalancers. We write a generalized depth $d$ recurrence with constant coefficients $x_{n-1}, x_{n-2}, \ldots, x_{n-d}$ plus a constant term $x_{0}$ as 
\begin{equation} \label{HI3} 
    c_{n} \ = \  x_{n-1} c_{n-1}+x_{n-2} c_{n-2}+\cdots +x_{n-d} c_{n-d}+x_{0} \ \longrightarrow\ (x_{n-1}, x_{n-2}, \ldots, x_{n-d},\underline{x_{0}})
\end{equation}
\begin{center} 
or $(x_{n-1}, x_{n-2}, \ldots, x_{n-d})$ if $x_{0}=0$.  \\
\end{center}
Note that the constant term $x_{0}$ will appear underlined if it is nonzero. 

%\alicia{Consider using different coefficients instead of $x, y, \ldots, z$ because these coefficients tell us nothing about the term which they are multiplying. It would be better to have coefficients such as $x_{n-1}, x_{n-2}, \ldots, x_{n-d}, x_0$ so that you know which term each coefficient is multiplying and to avoid confusion. Also, might be helpful to explicitly mention that the constant term will show up underlined, because that's easy to miss}
Although $(1,1)$ balancing and cobalancing numbers can both be described by depth two recurrences, we found that $(a,b)$ balancing and cobalancing numbers can be described more generally by depth five recurrences that follow interesting patterns. Similarly, the corresponding balancers and cobalancers can also be generally described by depth five recurrences. 
\scriptsize
\begin{table}[h!]
\centering
\begin{tabular}{|c|c|c|c|c|c|c|} 
\hline
\multicolumn{7}{|c|}{\begin{tabular}[c]{@{}c@{}}$b$\\\end{tabular}}                                                                                               \\ 
\hline
\multirow{8}{*}{$a$} &     & $1$                     & $2$                       & $3$                     & $4$                       & $5$                      \\ 
\cline{2-7}
                     & $1$ & $(1, 34, -34, -1, 1)$   & $(1, 98, -98, -1, 1)$     & $(1, 194, -194, -1, 1)$ & $(1, 322, -322, -1, 1)$   & $(1, 482, -482, -1, 1)$  \\ 
\cline{2-7}
                     & $2$ & $(1, 194, -194, -1, 1)$ & $(1, 34, -34, -1, 1)$     & $(1, 62, -62, -1, 1)$   & $(1, 98, -98, -1, 1)$     & $(1, 142, -142, -1, 1)$  \\ 
\cline{2-7}
                     & $3$ & $(1, 2, -2, -1, 1)$     & Undetermined              & $(1, 34, -34, -1, 1)$   & $(1, 254, -254, -1, 1)$ & Undetermined             \\ 
\cline{2-7}
                     & $4$ & $(1, 322, -322, -1, 1)$ & $(1, 194, -194, -1, 1)$   & Undetermined            & $(1, 34, -34, -1, 1)$     & Undetermined             \\ 
\cline{2-7}
                     & $5$ & $(1, 98, -98, -1, 1)$   & $(1, 898, -898, -1, 1)$   & Undetermined            & Undetermined              & $(1, 34, -34, -1, 1)$    \\ 
\cline{2-7}
                     & $6$ & $(1, 254, -254, -1, 1)$ & $(1, 2, -2, -1, 1)$       & $(1, 194, -194, -1, 1)$ & Undetermined              & Undetermined             \\ 
\cline{2-7}
                     & $7$ & $(1, 34, -34, -1, 1)$   & $(1, 1154, -1154, -1, 1)$ & Undetermined            & Undetermined              & Undetermined             \\
\hline
\end{tabular}
\caption{Depth Five Recurrences for $(a,b)$ Balancing Numbers and Balancers.
%\irfan{I like to keep tables as separate tex files and then add them and edit them directly when necessary. You are free to use my code as a teamplate and then edit your tables accordingly. I would also urge you to have better labels for your tables: table1 is not an instructive label\cdots  Also, in this table, please double check whether or not there should be a minus at $254$ in for $(3,4)$}
}
\label{table:depth5balancingrecurrences}
\end{table}
\normalsize

\footnotesize
\begin{table}[h!]
\centering
\begin{tabular}{|c|c|c|c|c|c|c|} 
\hline
\multicolumn{7}{|c|}{\begin{tabular}[c]{@{}c@{}}$b$\\\end{tabular}}                                                                                               \\ 
\hline
\multirow{8}{*}{$a$} &     & $1$                     & $2$                       & $3$                     & $4$                       & $5$                      \\ 
\cline{2-7}
                     & $1$ & $(1, 34, -34, -1, 1)$   & $(1, 98, -98, -1, 1)$     & $(1, 194, -194, -1, 1)$ & $(1, 322, -322, -1, 1)$   & $(1, 482, -482, -1, 1)$  \\ 
\cline{2-7}
                     & $2$ & $(1, 14, -14, -1, 1)$ & $(1, 34, -34, -1, 1)$     & $(1, 62, -62, -1, 1)$   & $(1, 98, -98, -1, 1)$     & $(1, 142, -142, -1, 1)$  \\ 
\cline{2-7}
                     & $3$ & None     &$(1, 34, -34, -1, 1)$& $(1, 34, -34, -1, 1)$   & $(1, 254, -254, -1, 1)$ & Undetermined             \\ 
\cline{2-7}
                     & $4$ & $(1, 18, -18, -1, 1)$ & $(1, 14, -14, -1, 1)$   & Undetermined            & $(1, 34, -34, -1, 1)$     & $(1, 42, -42, -1, 1)$  \\ 
\cline{2-7}
                     & $5$ & $(1, 10, -10, -1, 1)$   & $(1, 30, -30, -1, 1)$   & $(1, 98, -98, -1, 1)$   & Undetermined              & $(1, 34, -34, -1, 1)$    \\ 
\cline{2-7}
                     & $6$ & $(1, 16, -16, -1, 1)$ & None       & $(1, 14, -14, -1, 1)$ &  $(1, 34, -34, -1, 1)$  &  $(1, 178, -178, -1, 1)$            \\ 
\cline{2-7}
                     & $7$ & $(1, 34, -34, -1, 1)$   & $(1, 34, -34, -1, 1)$ & Undetermined            & Undetermined              & Undetermined             \\
\hline
\end{tabular}
\caption{Depth Five Recurrences for $(a,b)$ Cobalancing Numbers and Cobalancers.}
\label{table:depth5cobalancingrecurrences}
\end{table}
\normalsize

From Table \ref{table:depth5balancingrecurrences}, we see that most $(a,b)$ balancing and cobalancing numbers in this range can be described by depth five recurrence relations of the form $(1,K,-K,-1,1)$, where $K$ is a positive integer. The $(a,b)$ balancers and cobalancers can be described by the exact same recurrences as the corresponding balancing and cobalancing numbers, although the initial values in the sequences are different. One notable exception in Table \ref{table:depth5cobalancingrecurrences} is for $(3,1)$, for which there are no cobalancing numbers and is therefore marked ``None.'' %\irfan{When using quotations in latex, use ``(quoted text)'' instead of "(quoted text)"}
For several other sets of coefficients we were unable to find enough terms to determine a recurrence relation for the sequence, and these were marked ``Undetermined.'' This became a more significant issue for larger values of $a$ and $b$ for which the balancing and cobalancing numbers are often more spread out.

Some recurrence relations are the same for equivalent coefficients $(a,b)$ between balancing and cobalancing numbers. Interestingly, when the recurrence relations differed for a given set of coefficients, we found that the recurrences for balancing numbers and their balancers were of the form $\left(1,K^2-2,-\left(K^2-2\right),-1,1\right)$, while the recurrences for cobalancing numbers and their cobalancers were of the form $(1,K,-K,-1,1)$, for some positive integer value of $K$. However, the set of coefficients $(1,1)$ was unique because the terms in the balancing and cobalancing sequences, or the sequences for the balancers and cobalancers, were equivalent. 
%\irfan{Fix the tables; add actual rows with borders for $a$ and $b$}

The first two rows of Table \ref{table:depth5cobalancingrecurrences} can also be written as equivalent depth two recurrences plus a constant term for $(a,b)$ cobalancing numbers. 

\begin{table}[h!]
\begin{tabular}{|c|c|c|c|c|c|c|} 
\hline
\multicolumn{7}{|c|}{\begin{tabular}[c]{@{}c@{}}$b$\\\end{tabular}}                                                                                               \\ 
\hline
\multirow{3}{*}{$a$} &     & $1$                     & $2$                       & $3$                     & $4$                       & $5$                      \\ 
\cline{2-7}
                     & $1$ & $(6, -1, \underline{2})$   & $(10, -1, \underline{4})$     & $(14, -1, \underline{6})$ & $(18, -1, \underline{8})$   & $(22, -1, \underline{10})$  \\ 
\cline{2-7}
                     & $2$ & $(4, -1, \underline{1})$ & $(6, -1, \underline{2})$     & $(8, -1, \underline{3})$   & $(10, -1, \underline{4})$     & $(12, -1, \underline{5})$  \\ 
\hline
\end{tabular}
\caption{Depth Two Recurrences for $(a,b)$ Cobalancing Numbers}
\label{table:depth2cobalancingrecurrences}
\end{table}

From Table \ref{table:depth2cobalancingrecurrences} there is a more noticeable pattern in the recurrence relations. Following the proof technique used in \cite{BP} and \cite{PR} to derive the recurrence formulas for $(1,1)$ balancing and cobalancing numbers, we prove the following theorem.\\
\noindent\textbf{Theorem \ref{gen-cobalancing}}\textit{
    All $(a,b)$ cobalancing numbers such that $a \in \{1,2\}$ can be described by a depth two recurrence plus a constant term of the form 
\begin{equation}\label{BI1copied}
    c_{n} \ = \  (2m+2) c_{n-1}- c_{n-2}+m,
\end{equation}
where $m=2b/a$. For any of the depth two recurrences of this form, the sequence of $(a,b)$ cobalancing numbers starts with $c_{1}=m, c_{2}=2m^2+3m$. }

\begin{proof}
By the definition for cobalancing numbers with coefficients $(a, b)$, for a cobalancing number $n$ we have
\begin{eqnarray} \label{eqn1}
    a\left(\frac{(n+1)n}{2}\right) & \ = & \ b\left(\frac{(n+r)(n+r+1)}{2}-\frac{(n+1)n}{2}\right)  \\
    \frac{a}{b}(n+1)n& \ = & \ (n+r)(n+r+1)-(n+1)n \nonumber \\
\frac{a}{b}n^2+\frac{a}{b}n & \ = & \ n^2+r^2+2n r+n+r-n^2-n \nonumber \\
0 & \ = & \ r^2+(2n+1)r-\frac{a}{b}n^2-\frac{a}{b}n \nonumber \\
 r & \ = & \ \frac{-(2n+1)+ \sqrt{(4+4\frac{a}{b})n^2+(4+4\frac{a}{b})n+1}}{2}, \nonumber
\end{eqnarray}
where the negative solution of $r$ is omitted due to the range of $r$.
We will use the following two lemmas to complete our proof of Theorem \ref{gen-cobalancing}.

\begin{lem}\label{GClem1}
The smallest $(a,b)$ cobalancing number for $a\in \{ 1,2\}$ and $b \in \mathbb{Z}^{+}$ is $2b/a$ with cobalancer $1$.
\end{lem}
\begin{proof}
Since 
\begin{eqnarray}\label{eqn4}
    a\left(1+\cdots+\frac{2b}{a}\right) \ & = & \  
    a\left( \frac{\frac{2b}{a}\left(\frac{2b}{a}+1\right)}{2}\right) \nonumber\\
    \ & = & \ \frac{2b^{2}}{a}+b\nonumber \\
    \ & = & \ b\left(\frac{2b}{a}+1\right),
\end{eqnarray}
we have that $2b/a$ is a cobalancing number with cobalancer 1. Since by Equation \eqref{eqn1}, the value of $r$ is uniquely determined by $n$ for fixed $a$ and $b$, and $r$ strictly increases as $n$ increases, if there exists a cobalancing number $c<2b/a$, its cobalancer must be less than 1 which contradicts the possible range of the cobalancer. Hence $2b/a$ must be the smallest cobalancing number. 
\end{proof}

\begin{lem}\label{GClem2}
Let
\begin{equation}\label{eqn2}
    f(x)\ =\  \left (\frac{2b}{a}+1 \right )x+\frac{b}{a}\sqrt{\left( 4+4\frac{a}{b}\right)x^{2}+\left( 4+4\frac{a}{b}\right)x+1} +\frac{b}{a}.   
\end{equation}
    If $x$ is an $(a,b)$ cobalancing number with $a\in \{ 1,2\}$ and $b \in \mathbb{Z}^{+}$, then there is no cobalancing number $y$ such that $x<y<f(x)$.
\end{lem}
\begin{proof}
Since the derivative of $f$, 

\begin{equation} \label{eqn3}
    f'(x)\ =\ \frac{2\sqrt{b}\left(2x+1\right)\left(b+a\right)+2b\sqrt{4bx^{2}+4ax^{2}+4ax+4bx+b}}
    {a\sqrt{4bx^{2}+4ax^{2}+4ax+4bx+b}}+1 \ >\ 0
\end{equation}
for non-negative $x$, $f$ strictly increases for $x$ of that range. Hence the range of $f(x)$ over non-negative $x$ is $[f(0),\infty) = \left[2b/a, \infty \right)$. Also, since $f$ is bijective and $x < f(x)$ for all $x\geq 0$, $f^{-1}$, which is defined over $\left[2b/a, \infty \right)$, exists and is strictly increasing with $f^{-1}(x)<x$. \\
\indent By Lemma \ref{GClem1}, for a cobalancing number $x$, we have $x\geq 2b/a$. Thus let $u=f^{-1}(x)$, then $f(u)=x$ and 
\begin{equation} \label{eqn5}
    u\ =\ \left (\frac{2b}{a}+1 \right )x-\frac{b}{a}\sqrt{\left( 4+4\frac{a}{b}\right)x^{2}+\left( 4+4\frac{a}{b}\right)x+1} +\frac{b}{a}.
\end{equation}
\indent Note that $u$ is an integer if $a \in \{1,2\}$. We will show that $u$ is also a cobalancing number by showing that its cobalancer, given by Equation \eqref{eqn1},
\begin{equation} \label{eqn6}
     r'\ =\  \frac{-(2u+1)+ \sqrt{(4+4\frac{a}{b})u^2+(4+4\frac{a}{b})u+1}}{2},
\end{equation}
is an integer. Since $f(u)=x$ we have
\begin{equation}\label{eqn7}
    x\ =\ \left (\frac{2b}{a}+1 \right )u+\frac{b}{a}\sqrt{\left( 4+4\frac{a}{b}\right)u^{2}+\left( 4+4\frac{a}{b}\right)u+1} +\frac{b}{a},
\end{equation}
and thus 
\begin{eqnarray}\label{eqn8}
    \frac{b}{a}\sqrt{\left( 4+4\frac{a}{b}\right)u^{2}+\left( 4+4\frac{a}{b}\right)u+1} \ & = & \ x- \frac{b}{a}-\left(\frac{2b}{a}+1\right)u\nonumber\\
    \sqrt{\left( 4+4\frac{a}{b}\right)u^{2}+\left( 4+4\frac{a}{b}\right)u+1} \ & = & \ \left(x-u\right)\frac{a}{b}-2u-1. 
\end{eqnarray}
Substituting Equation \eqref{eqn5} for $u$, we get
\begin{equation}\label{eqn9}
    \sqrt{\left( 4+4\frac{a}{b}\right)u^{2}+\left( 4+4\frac{a}{b}\right)u+1}
    = \left( \sqrt{\left( 4+4\frac{a}{b}\right)x^{2}+\left( 4+4\frac{a}{b}\right)x+1}-\left(2x+1\right)\right) -2u-1.
\end{equation}
Hence
\begin{eqnarray}\label{eqn10}
    r'\ & = & \ \frac{-(2u+1)-2u-1-(2x+1)+\sqrt{\left( 4+4\frac{a}{b}\right)x^{2}+\left( 4+4\frac{a}{b}\right)x+1}}{2}\nonumber \\
    \ & = & \ -(2u+1)+r,
\end{eqnarray}
where $r$ is the cobalancer of $x$, so $r'$ must be an integer. \\
\indent We define the sequence $(c_{n})$ by $c_{0}=0$ and $c_{n}=f(c_{n-1})$. Suppose for contradiction there exists a cobalancing number $c'$ between $c_{i}$ and $c_{i+1}$. Thus we have 
\begin{eqnarray}\label{eqn11}
    c_{i} \  < & \ c' \ & <    \ c_{i+1} \nonumber \\
    c_{i-1} \  < & \ f^{-1}(c') \ & <  \ c_{i} \nonumber \\
    c_{i-2} \  < & \ f^{-2}(c') \ & <  \ c_{i-1} \nonumber \\
    & \vdots& \nonumber\\
    c_{0} \  < & \ f^{-i}(c') \ & <  \ c_{1}.
\end{eqnarray}
By our previous result, $f^{-i}(c')$ must also be a cobalancing number. However, Equation \eqref{eqn11} contradicts Lemma \ref{GClem1}, hence our assumption is false and the sequence starting with $c_1 = 2b/a$ determined by $f$ are the only cobalancing numbers.
\end{proof}
We now continue our proof of Theorem \ref{gen-cobalancing}.
From Equations \eqref{eqn5} and \eqref{eqn7} we have that 
\begin{equation}\label{eqn12}
    c_{n+1}\ =\ \left (\frac{2b}{a}+1 \right )c_n+\frac{b}{a}\sqrt{\left( 4+4\frac{a}{b}\right)c_n^{2}+\left( 4+4\frac{a}{b}\right)c_n+1} +\frac{b}{a},
\end{equation}
and 
\begin{equation}\label{eqn13}
    c_{n-1}\ =\ \left (\frac{2b}{a}+1 \right )c_n-\frac{b}{a}\sqrt{\left( 4+4\frac{a}{b}\right)c_n^{2}+\left( 4+4\frac{a}{b}\right)c_n+1} +\frac{b}{a},
\end{equation}
where $c_n$ is the $n$\textsuperscript{{\rm th}} cobalancing number. Adding both equations gives 
\begin{equation}\label{eqn14}
    c_{n+1} + c_{n-1}\ =\ 2\left (\frac{2b}{a}+1 \right )c_n+\frac{2b}{a}.
\end{equation}
Let $m= 2b/a$. Then,
\begin{equation}\label{eqn15}
     c_{n+1} \ =\ (2m+2) c_{n}- c_{n-1}+m.
\end{equation}
\end{proof}

The depth two recurrences are more useful than the depth five recurrences which do not describe how the sequences begin, and can only be used to determine the sequence after the initial terms are indicated. 
\begin{table}[h!]
\begin{tabular}{|c|c|c|c|c|c|c|} 
\hline
\multicolumn{7}{|c|}{\begin{tabular}[c]{@{}c@{}}$b$\\\end{tabular}}                                                                                               \\ 
\hline
\multirow{3}{*}{$a$} &     & $1$                     & $2$                       & $3$                     & $4$                       & $5$                      \\ 
\cline{2-7}
                     & $1$ & $(6, -1)$   & $(10, -1)$     & $(14, -1)$ & $(18, -1)$   & $(22, -1)$  \\ 
\cline{2-7}
                     & $2$ & $(4, -1)$ & $(6, -1)$     & $(8, -1)$   & $(10, -1)$     & $(12, -1)$  \\ 
\hline
\end{tabular}
\caption{Depth Two Recurrences for $(a,b)$ Cobalancers.}
\label{table:depth2cobalancerrecurrences}
\end{table}
The corresponding cobalancers can also be expressed as depth two recurrences but without the constant term. Their recurrence relations are described by the following theorem.\\
\textbf{Theorem \ref{gen-cobalancer}}\textit{
The cobalancers of $(a,b)$ cobalancing numbers such that $a \in \{1,2\}$ can be described by a depth two recurrence of the form
\begin{equation}\label{BI2copied}
    r_{n} \ = \ (2m+2) r_{n-1}- r_{n-2},
\end{equation}
where $r_{n}$ is the cobalancer of the $n$\textsuperscript{{\rm th}} cobalancing number and $m=2b/a$. For any of the depth two recurrences of this form, the sequence of $(a,b)$ cobalancers starts with $r_{1}=1, r_{2}=2m+2$. }

\begin{proof}
From Equation \eqref{eqn10} we have $2c_{n}+1=r_{n+1}-r_{n}$ and hence $c_{n}=\frac{r_{n+1}-r_{n}-1}{2}$. Substituting this result into Theorem \ref{gen-cobalancing}, we get
\begin{eqnarray}\label{eqn16}
    \frac{r_{n+2}-r_{n+1}-1}{2} \ & = & \ (m+1)(r_{n+1}-r_{n}-1)-\frac{r_{n}-r_{n-1}-1}{2} +m \nonumber\\
    (r_{n+2}-r_{n+1}) \ & = & \  (2m+2)(r_{n+1}-r_{n}) - (r_{n}-r_{n-1}).
\end{eqnarray}
We now use induction to prove that $r_{n}= (2m+2) r_{n-1}- r_{n-2}$. \\
\indent \emph{Base Case:} When $n=3$, 
\begin{eqnarray}\label{eqn17}
    r_{3} \ & = & \ 4m^{2}+8m+3 \nonumber\\
    \ & = & \ (2m+2)^{2} -1 \nonumber\\
    \ & = & \ (2m+2)r_{2}-r_{1}.
\end{eqnarray}
\indent \emph{Induction Step:} Assume our result holds when $n=k$, hence
\begin{equation}\label{eqn0}
    r_{k} \ = \  (2m+2) r_{k-1}- r_{k-2}.
\end{equation}
By Equation \eqref{eqn16} we have 
\begin{equation}\label{eqn18}
    r_{k+1}-r_{k}  =    (2m+2)(r_{k}-r_{k-1}) - (r_{k-1}-r_{k-2}),
\end{equation}
and adding Equations \eqref{eqn0} and \eqref{eqn18} gives
\begin{equation}\label{eqn19}
    r_{k+1} \ = \  (2m+2) r_{k}- r_{k-1},
\end{equation}
which proves our result for $n=k+1$. This only holds for $a \in \{1,2\}$ since we used results from Theorem \ref{gen-cobalancing} which required this condition.
\end{proof}

The $(a,b)$ balancing numbers and balancers, however, could only be generally expressed by the depth five recurrences given above for a larger range of $a$. %\irfan{Wouldn't be bad to put one of these expressions into an equation environment so you can directly reference the result instead of saying that something is just above}
%\irfan{I've also noticed you all use footnotesize often when creating a table or something similar. In general, the scalebox command is your friend}

%\Bella{same question here}

%%%%%%%%%%%%%%%%%%%%%%%%%%%%%%%%%%%%%%%%%%%%%%%%%%%%%%%%%%%%%%%%%%%%%%%%%%%%%%%%%%%%%%%%%%%%%%%%%%%%%%%%%%%%%%%%%%%%%%%%%%%%%%%%%%%%%%%%%%
%%%%%%%%%%%%%%%%%%%%%%%%%%%%%%%%%%%%%%%%%%%%%%%%%%%%%%%%%%%%%%%%%%%%%%%%%%%%%%%%%%%%%%%%%%%%%%%%%%%%%%%%%%%%%%%%%%%%%%%%%%%%%%%%%%%%%%%%%%
%%%%%%%%%%%%%%%%%%%%%%%%%%%%%%%%%%%%%%%%%%%%%%%%%%%%%%%%%%%%%%%%%%%%%%%%%%%%%%%%%%%%%%%%%%%%%%%%%%%%%%%%%%%%%%%%%%%%%%%%%%%%%%%%%%%%%%%%%%
\section{Reciprocal Sum of Sequences}
\subsection{Reciprocal Sums of (a, b) Cobalancing Numbers}
We now derive formulas for reciprocal sums of depth two recurrences, which we then apply to the recurrences for cobalancing numbers we found above. \\
\\
\textbf{Theorem \ref{thm:II0}.}
\textit{
For all recurrences of the form
\begin{equation}\nonumber
c_{n+1} \ =  \ q c_{n}-c_{n-1}+s,
\end{equation}
where $q, s \in \mathbb{R}_{\ne 0}, q \geq 2, c_{0}=0, c_{1}=s$, if $s > \frac{1}{2}$ and
\begin{equation}\nonumber
0 \ \leq \ (q-s) c_{n} - 2c_{n-1} +s-1,
\end{equation}
then for any positive integer $n$, we have
\begin{equation}\nonumber
\left\lfloor\left(\sum_{k = n}^{\infty} \frac{1}{c_{k}}\right)^{-1}\right\rfloor \ = \ c_n-c_{n-1}-1,
\end{equation}
where $c_{n}$ is the $n$\textsuperscript{{\rm th}} term in the sequence.}
\begin{proof}

We begin by proving the following lemmas.

\begin{lem} \label{lem: II0}
For a recurrence relation in the form described in Theorem \ref{thm:II0}, for all $n \geq 1$, 
\begin{equation}\label{II4}
    c_{n}^2\ =\ c_{n+1}c_{n-1}+s c_{n}.
\end{equation}

\begin{proof}

We  prove by induction that Equation \eqref{II4} holds for all $n \geq 1$. \\ \

\emph{Base Case:} When $n=1$, both sides of Equation \eqref{II4} gives $s^2$, so \eqref{II4} is true for $n=1$. \\ \

\emph{Induction Step:} Suppose Equation \eqref{II4} is true for $n=k$. Then,
\begin{eqnarray}\label{II5}
    c_{k+2}c_{k}+s c_{k+1} & \ =\  & (q c_{k+1}-c_{k}+s)c_{k}+s c_{k+1} \nonumber\\
    & = & (q^2 c_{k}-q c_{k-1}+q s-c_{k}+s)c_{k}+s c_{k+1}\nonumber\\
    & = & q^2 c_{k}^2-q
    c_{k}c_{k-1}-c_{k}^2+c_{k}q s+c_{k}s+s (q c_{k}-c_{k-1}+s)\nonumber\\
    & = & q^2 c_{k}^2-q
    c_{k}c_{k-1}-c_{k}^2+2s q c_{k}+s c_{k}-s c_{k-1}+s^2\nonumber\\
    & = & q^2 c_{k}^2-q
    c_{k}c_{k-1}-(c_{k+1}c_{k-1}+s c_{k})+2s q c_{k}+s c_{k}-s c_{k-1}+s^2 \nonumber\\ 
    & = & q^2 c_{k}^2-q
    c_{k}c_{k-1}-(q c_{k}-c_{k-1}+s)c_{k-1}+2s q c_{k}-s c_{k-1}+s^2 \nonumber\\ 
    & = & q^2 c_{k}^2-2q
    c_{k}c_{k-1}+c_{k-1}^2-2s c_{k-1}+2s q c_{k}+s^2 \nonumber\\ 
    & = & (q c_{k}-c_{k-1}+s)^2 \nonumber\\
    & = & c_{k+1}^2. 
\end{eqnarray}
Thus, Equation \eqref{II4} holds for $n=k+1$, completing the proof. 
\end{proof}
\end{lem} 

\begin{lem}\label{IIlemma}
For all recurrences of the form
\begin{equation}\nonumber
c_{n+1} \ =  \ q c_{n}-c_{n-1}+s,
\end{equation}
where $q,s \in \mathbb{R}$, $s > 0$, $q \geq 2$, and $c_{0}=0, c_{1}=s$, we have
\begin{equation}\nonumber
\lim_{n \to \infty}\frac{1}{c_{n} - c_{n-1}} \ = \ 0.
\end{equation}
\end{lem}
\begin{proof}
The limit in question must be non-negative since $c_i > c_{i-1}$. This is proved inductively by noticing that $c_1 > c_0$ and if $c_n > c_{n-1}$, then $c_{n+1} \geq 2c_n - c_{n-1} + s > c_n + s > c_n$. Then
\begin{equation}\nonumber
\begin{split}\label{IIlemmaproof}
0 \ \leq \ \lim_{n \to \infty}\frac{1}{c_{n} - c_{n-1}}  & \ = \ \lim_{n \to \infty}\frac{1}{(q-1)c_{n-1} - c_{n-2} + s} \\
&\ \leq \ \lim_{n \to \infty}\frac{1}{c_{n-1} - c_{n-2} + s} \\
 & \ \leq \  \lim_{n \to \infty}\frac{1}{(q-1)c_{n-2} - c_{n-3} + 2s} \\
 &\  \ \vdots \ \\ 
 &\ \leq \ \lim_{n \to \infty}\frac{1}{c_{1} - c_{0} + (n-1)s} \\
 &\ = \ \lim_{n \to \infty}\frac{1}{ns} \\
 &\ = \ 0.
\end{split}
\end{equation}
\end{proof}

We now use these results to prove Theorem \ref{thm:II0}. Let $s > \frac{1}{2}$ and
\begin{equation}\label{II53}
0 \ \leq \ (q-s) c_{n} - 2c_{n-1} +s-1.
\end{equation}
Note that
\begin{equation}
\begin{split}\label{II6}
\frac{1}{c_{n}-c_{n-1}} - \frac{1}{c_n}  & \ = \ \frac{c_{n-1}}{c_n^{2}-c_{n}c_{n-1}} \\
&\ = \ \frac{c_{n-1}}{c_{n-1}(c_{n+1}-c_{n})+s c_{n}} \\
 & \ <\  \frac{1}{c_{n+1}-c_{n}}.
\end{split}
\end{equation}
Therefore we have that
\begin{equation}\label{II7}
\frac{1}{c_{n}-c_{n-1}} \ < \  \frac{1}{c_{n+1}-c_{n}} + \frac{1}{c_n}
\end{equation}
and
\begin{equation}\label{II8}
\frac{1}{c_{n+1}-c_{n}} \ < \  \frac{1}{c_{n+2}-c_{n+1}} + \frac{1}{c_{n+1}}
\end{equation}
and
\begin{equation}\label{II9}
\frac{1}{c_{n+2}-c_{n+1}} \ < \ \frac{1}{c_{n+3}-c_{n+2}} + \frac{1}{c_{n+2}}.
\end{equation}
Adding inequalities \eqref{II7}, \eqref{II8} and \eqref{II9}, we have
\begin{equation}\label{II10}
\frac{1}{c_{n}-c_{n-1}} \ < \  \frac{1}{c_n} +\frac{1}{c_{n+1}} + \frac{1}{c_{n+2}} + \frac{1}{c_{n+3}-c_{n+2}}. 
\end{equation}
Continuing the pattern and noting that the rightmost term converges to $0$ as $n$ approaches infinity by Lemma \ref{IIlemma}, we obtain
\begin{equation}\label{II11}
\frac{1}{c_{n}-c_{n-1}} \ < \  \sum_{k = n}^{\infty} \frac{1}{c_{k}}.
\end{equation}

Next, suppose for contradiction that
\begin{equation}\label{II12}
\frac{1}{c_{n}-c_{n-1}-1} \ <\  \frac{1}{c_{n+1}-c_{n}-1} + \frac{1}{c_{n}}. 
\end{equation}
Note that $c_{n}-c_{n-1}-1 > 0$. This can be proven inductively by using the recurrence relations and the fact that $q \geq 2$. For the base case we get $c_2 - c_1 - 1 \geq 2s - 1 > 0$ since $s > \frac{1}{2}$. Then, the inductive step gives us $c_n - c_{n-1} - 1 \geq c_{n-1} - c_{n-2} + s - 1 > s > 0$.
Thus, we have 
\begin{eqnarray}\label{II13}
\frac{1}{c_{n}-c_{n-1}-1} \ &\ <\ &\ \frac{c_{n+1}-1}{c_{n+1}c_{n}-c_{n}^2-c_{n}}
\nonumber\\
c_{n}c_{n+1}-c_{n}^2-c_{n} \ &\ <\ &\ c_{n+1}c_{n}-c_{n}-c_{n+1}c_{n-1}+c_{n-1}-c_{n+1}+1
\nonumber\\
c_{n}^2 \ &\ >\ &\ c_{n+1}c_{n-1}-c_{n-1}+c_{n+1}-1
\nonumber\\
c_{n+1}c_{n-1}+s c_{n} \ &\ >\ &\ c_{n+1}c_{n-1}-c_{n-1}+c_{n+1}-1
\nonumber\\
0 \ &\ >\ &\ c_{n+1}-s c_{n}-c_{n-1}-1
\nonumber\\
0 \ &\ >\ &\ (q-s) c_{n} - 2c_{n-1} +s-1,
\end{eqnarray}
which contradicts Equation \eqref{II53}.  Thus
\begin{equation}\label{II14}
\frac{1}{c_{n}-c_{n-1}-1} \ \geq\ \frac{1}{c_{n}} + \frac{1}{c_{n+1}-c_{n}-1}. 
\end{equation}
Continuing the pattern, we obtain
\begin{equation}\label{II15}
\frac{1}{c_{n}-c_{n-1}-1} \ \geq\  \sum_{k = n}^{\infty} \frac{1}{c_{k}},
\end{equation}
and we obtain the non-strict inequality because $\displaystyle{\lim_{n \to \infty}\frac{1}{c_{n} - c_{n-1} - 1} \ = \ 0}$, which can be shown by following the proof of Lemma \ref{IIlemma}.
Combining inequalities \eqref{II11} and \eqref{II15}, we have
\begin{equation}\label{II16}
\frac{1}{c_{n}-c_{n-1}}  \ < \  \sum_{k = n}^{\infty} \frac{1}{c_{k}}\ \leq\ \frac{1}{c_{n}-c_{n-1}-1}.
\end{equation}
Therefore
\begin{equation}\label{II17}
\left\lfloor\left(\sum_{k = n}^{\infty} \frac{1}{c_{k}}\right)^{-1}\right\rfloor \ = \ c_n-c_{n-1}-1,
\end{equation}
which completes the proof.
\end{proof}

\begin{cor} \label{cor: II0}
For $(a,b)$ cobalancing numbers with $a \in \{1,2\}$, for any positive integer $n$ we have
\begin{equation}\label{thm:II18}
\left\lfloor\left(\sum_{k = n}^{\infty} \frac{1}{c_{k}}\right)^{-1}\right\rfloor \ = \ c_n-c_{n-1}-1,
\end{equation}
where $c_{n}$ is the $n$\textsuperscript{{\rm th}} cobalancing number.
\end{cor}

\begin{proof}
By Theorem \ref{gen-cobalancing}, $(a,b)$ cobalancing numbers with $a \in \{1,2\}$ satisfy recurrence relations of the form
\begin{equation}\label{II19}
c_{n+1} \ = \ (2m+2)c_{n}-c_{n-1}+m,
\end{equation}
where $m = 2b/a$, and the sequences start with $c_{1}=m, c_{2}= 2m^2+2m$. Note that although the initial conditions are slightly different from those described in Theorem \ref{thm:II0}, the sequences are simply shifted so we can extend them to include $c_0 = 0$. It is simple to check that these recurrences satisfy the conditions in Theorem \ref{thm:II0}, and the result directly follows. 
\end{proof} 

\noindent \textbf{Theorem \ref{thm:II1}}
\textit{
For all recurrences of the form
\begin{equation}\nonumber
c_{n+1} \ = \ q c_{n}+r c_{n-1},
\end{equation}
%
% FIX - Make changes here. we need q > 2 and r >= -1,
% or q > 3 and r >= 0. We need to consider whether r < 0
% for the result: when r < 0 we always have t^2(-r)^(n-1)
% is positive, so the result wont depend on the parity 
% of n. when r >= 0 we need to consider whether
% t^2(-r)^(n-1) is positive or negative (which we have 
% done)
%
where $q, r \in \mathbb{R}_{\ne 0}$, $c_{0}=0$, $c_{1} = t$, and $t>0$, we have the following cases.}\\

\emph{Case 1: when $q\geq 3$ and  $-1\leq r<0$, if 
\begin{equation*}
    t^2(-r)^{n-1} \ \leq\  c_{n+1}-c_{n-1}-1, 
\end{equation*}
then for any positive integer $n$,
\begin{equation}\nonumber
\left\lfloor\left(\sum_{k = n}^{\infty} \frac{1}{c_{k}}\right)^{-1}\right\rfloor \ = \ c_{n}-c_{n-1}-1.
\end{equation}
}

\emph{Case 2: when $q\geq 2$ and $r\geq 0$, if 
\begin{equation}\nonumber
\begin{cases}
t^2(-r)^{n-1} \ \leq\  c_{n+1}-c_{n-1}-1 & \text{{\rm if $n$ is odd}}\\
t^2(-r)^{n-1} \ >\   -c_{n+1}+c_{n-1}-1  & \text{{\rm if $n$ is even}},\\
\end{cases}
\end{equation}
then for any positive integer $n$, we have
\begin{equation}\nonumber
\left\lfloor\left(\sum_{k = n}^{\infty} \frac{1}{c_{k}}\right)^{-1}\right\rfloor \ = \ 
\begin{cases}
c_{n}-c_{n-1}-1 & \text{{\rm if $n$ is odd and $n$}} \geq 1\\
c_{n}-c_{n-1}  & \text{{\rm if $n$ is even and $n$}} \geq 2,\\
\end{cases}
\end{equation}
where $c_{n}$ is the $n$\textsuperscript{{\rm th}} term in the sequence.
}

\begin{proof}
We begin by proving the following lemma.

\begin{lem} 
For a recurrence relation in the form described in Theorem \ref{thm:II1}, for all $n \geq 1$, 
\begin{equation} \label{II23}
    c_{n}^2\ = \ c_{n+1}c_{n-1}+t^2(-r)^{n-1}.
\end{equation}
\end{lem} 

\begin{proof}
We  prove by induction that Equation \eqref{II23} holds or all $n \geq 1$. \\ \

\emph{Base Case:} When $n=1$, the left side of Equation \eqref{II23} is $t^2$, and the right side is $(0)(qt)+t^2 = t^2$. Both sides are equal so Equation \eqref{II23} is true for $n=1$. \\ \

\emph{Induction Step:} Suppose Equation \eqref{II23} is true for $n=k$. Then 
\begin{eqnarray}\label{II24}
    c_{k+2}c_{k}+t^2(-r)^{k} & = & (q c_{k+1}+r c_{k})c_{k}+t^2(-r)^{k} \nonumber\\
    & = & (q^2 c_{k}+q r c_{k-1}+r c_{k})c_{k}+t^2(-r)^{k} \nonumber\\
    & = & q^2 c_{k}^2+q r c_{k}c_{k-1}-r(-c_{k}^2+t^2(-r)^{k-1}) \nonumber\\
    & = & q^2 c_{k}^2+q r c_{k}c_{k-1}+r c_{k+1}c_{k-1} \nonumber\\
    & = & q^2 c_{k}^2+2q r c_{k}c_{k-1}+r^2 c_{k-1}^2 \nonumber\\ 
    & = & (q c_{k}+r c_{k-1})^2 \nonumber\\
    & = & c_{k+1}^2.
\end{eqnarray}
Thus, Equation \eqref{II23} holds for $n=k+1$, completing the proof.
\end{proof}

\begin{lem} \label{lem1}
For a recurrence relation in the forms described in Theorem \ref{thm:II1}, we have 
\begin{equation} \label{II55}
\lim_{n \to \infty} \frac{1}{c_{n}-c_{n-1}}\ =\ 0. \nonumber
\end{equation}
\end{lem}

\begin{proof}
We use strong induction to show that for all $n \in \mathbb{N}$ we have $c_{n}-c_{n-1}\geq nt$, first focusing on the case when $q \geq 3$ and $r \geq -1$.\\ 
\indent \emph{Base Case:} By definition, $c_{1}-c_{0}=t$. \\ \

\emph{Induction Step:} Suppose for all $i\in\{ 2,\cdots k\}$, that $c_{i-1}-c_{i-2}\geq (i-1)t$. Note from our inductive hypothesis it follows that the terms in the sequence $c_{n}$ strictly increase until $c_{k}$, and $c_{i} \geq t$ for all $i\in\{ 1,\cdots k\}$. Then,
\begin{eqnarray}
    c_{k+1}-c_{k} & \  = \  &(q-1)c_{k}+ rc_{k-1} \nonumber\\
     & \ \geq \ & 2c_{k} - c_{k-1}\nonumber\\
     & \ \geq \ & c_{k} + nt \nonumber\\
     & \ \geq \ & (n+1)t, \nonumber
\end{eqnarray}
finishing our induction. The same result follows easily when $q \geq 2$ and $r \geq 0$. Therefore
\begin{eqnarray}
\lim_{n \to \infty}\frac{1}{c_{n}-c_{n-1}}  & \  \leq \ & \lim_{n \to \infty}\frac{1}{nt} \nonumber \\
& \  \leq \ & 0, \nonumber \\
\end{eqnarray}
and the limit must be non-negative, so the result follows.
\end{proof}

We now use these results to prove Theorem \ref{thm:II1}. We will first prove \emph{Case 2}, and note that \emph{Case 1} will follow when we consider the case when $n$ is odd. Let
\begin{eqnarray} \label{II54}
t^2(-r)^{n-1} & \ \leq\ & c_{n+1}-c_{n-1}-1 \ \ \ \ \ \text{{\rm if $n$ is odd, and}} \\
t^2(-r)^{n-1} & \ >\ & -c_{n+1}+c_{n-1}-1  \ \ \  \text{{\rm if $n$ is even}}. \nonumber
\end{eqnarray}

First we consider the case when $n \geq 1$ and $n$ is odd. Note that, since $t^2(-r)^{n-1}$ is positive when $n$ is odd,
\begin{equation}
\begin{split}\label{II25}
\frac{1}{c_{n}-c_{n-1}} - \frac{1}{c_n}  & \ = \ \frac{c_{n-1}}{c_n^{2}-c_{n}c_{n-1}} \\
&\ =\  \frac{c_{n-1}}{c_{n-1}(c_{n+1}-c_{n})+t^2(-r)^{n-1}} \\
 & \ <\  \frac{1}{c_{n+1}-c_{n}}.
\end{split}
\end{equation}
Therefore we have that
\begin{equation}\label{II26}
\frac{1}{c_{n}-c_{n-1}} \ < \ \frac{1}{c_{n+1}-c_{n}} + \frac{1}{c_n}
\end{equation}
and
\begin{equation}\label{II27}
\frac{1}{c_{n+1}-c_{n}} \ < \ \frac{1}{c_{n+2}-c_{n+1}} + \frac{1}{c_{n+1}}
\end{equation}
and
\begin{equation}\label{II28}
\frac{1}{c_{n+2}-c_{n+1}} \ < \ \frac{1}{c_{n+3}-c_{n+2}} + \frac{1}{c_{n+2}}.
\end{equation}
Adding inequalities \eqref{II26}, \eqref{II27} and \eqref{II28}, we have
\begin{equation}\label{II29}
\frac{1}{c_{n}-c_{n-1}} \ <\  \frac{1}{c_n} +\frac{1}{c_{n+1}} + \frac{1}{c_{n+2}} + \frac{1}{c_{n+3}-c_{n+2}}. 
\end{equation}
Continuing the pattern and noting that the rightmost term converges to $0$ as $n$ approaches infinity by Lemma \ref{lem1}, we obtain
\begin{equation}\label{II30}
\frac{1}{c_{n}-c_{n-1}} \ < \  \sum_{k = n}^{\infty} \frac{1}{c_{k}}.
\end{equation}
\ \\ \
\indent Next, suppose, for contradiction, 
\begin{equation}\label{II31}
\frac{1}{c_{n}-c_{n-1}-1} \ <\  \frac{1}{c_{n+1}-c_{n}-1} + \frac{1}{c_{n}}. 
\end{equation}
Then we have 
\begin{eqnarray}\label{II32}
\frac{1}{c_{n}-c_{n-1}-1} \ &\ <\ &\ \frac{c_{n+1}-1}{c_{n+1}c_{n}-c_{n}^2-c_{n}}
\nonumber\\
c_{n}c_{n+1}-c_{n}^2-c_{n} \ &\ <\ &\ c_{n+1}c_{n}-c_{n}-c_{n+1}c_{n-1}-c_{n+1}+c_{n-1}+1
\nonumber\\
c_{n}^2 \ &\ >\ &\ c_{n+1}c_{n-1}-c_{n-1}+c_{n+1}-1
\nonumber\\
c_{n+1}c_{n-1}+t^2(-r)^{n-1} \ &\ >\ &\ c_{n+1}c_{n-1}-c_{n-1}+c_{n+1}-1
\nonumber\\
0 \ &\ >\ &\ c_{n+1}-c_{n-1}-t^2(-r)^{n-1}-1,
\end{eqnarray}
which contradicts Equation \eqref{II54}. Thus
\begin{equation}\label{II33}
\frac{1}{c_{n}-c_{n-1}-1} \ \geq\ \frac{1}{c_{n}} + \frac{1}{c_{n+1}-c_{n}-1}. 
\end{equation}
Continuing the pattern, we obtain
\begin{equation}\label{II34}
\frac{1}{c_{n}-c_{n-1}-1} \ \geq\  \sum_{k = n}^{\infty} \frac{1}{c_{k}},
\end{equation}
where we retain the non-strict inequality because $\displaystyle{\lim_{n \to \infty}\frac{1}{c_{n} - c_{n-1} - 1} \ = \ 0}$ which can be shown by following the proof of Lemma \ref{lem1}. Combining the two inequalities, we have
\begin{equation}\label{II35}
\frac{1}{c_{n}-c_{n-1}}  \ <\  \sum_{k = n}^{\infty} \frac{1}{c_{k}} \leq\  \frac{1}{c_{n}-c_{n-1}-1}.
\end{equation}
Therefore
\begin{equation}\label{II36}
\left\lfloor\left(\sum_{k = n}^{\infty} \frac{1}{c_{k}}\right)^{-1}\right\rfloor \ = \ c_n-c_{n-1}-1,
\end{equation}
which completes the proof for odd value of $n$. \\

Next we consider the case when $n \geq 2$ and $n$ is even. Note that, since $t^2(-r)^{n-1}$ is negative when $n$ is even,
\begin{equation}
\begin{split}\label{II37}
\frac{1}{c_{n}-c_{n-1}} - \frac{1}{c_n}  & \ = \ \frac{c_{n-1}}{c_n^{2}-c_{n}c_{n-1}} \\
& \ = \ \frac{c_{n-1}}{c_{n-1}(c_{n+1}-c_{n})+t^2(-r)^{n-1}} \\
 & \ >\  \frac{1}{c_{n+1}-c_{n}}.
\end{split}
\end{equation}
Therefore we have that
\begin{equation}\label{II38}
\frac{1}{c_{n}-c_{n-1}} \ > \  \frac{1}{c_{n+1}-c_{n}} + \frac{1}{c_n}
\end{equation}
and
\begin{equation}\label{II39}
\frac{1}{c_{n+1}-c_{n}} \ > \  \frac{1}{c_{n+2}-c_{n+1}} + \frac{1}{c_{n+1}}
\end{equation}
and
\begin{equation}\label{II40}
\frac{1}{c_{n+2}-c_{n+1}} \ > \ \frac{1}{c_{n+3}-c_{n+2}} + \frac{1}{c_{n+2}}.
\end{equation}
Adding inequalities \eqref{II38}, \eqref{II39} and \eqref{II40}, we have
\begin{equation}\label{II41}
\frac{1}{c_{n}-c_{n-1}} \ >\  \frac{1}{c_n} +\frac{1}{c_{n+1}} + \frac{1}{c_{n+2}} + \frac{1}{c_{n+3}-c_{n+2}}. \\
\end{equation}
Continuing the pattern, we obtain
\begin{equation}\label{II42}
\frac{1}{c_{n}-c_{n-1}} \ >\  \sum_{k = n}^{\infty} \frac{1}{c_{k}}.
\end{equation}

Next, suppose, for contradiction, 
\begin{equation}\label{II43}
\frac{1}{c_{n}-c_{n-1}+1} \ \geq \  \frac{1}{c_{n+1}-c_{n}+1} + \frac{1}{c_{n}}. 
\end{equation}
Then we have 
\begin{eqnarray}\label{II44}
\frac{1}{c_{n}-c_{n-1}+1} \ &\ \geq\ &\ \frac{c_{n+1}+1}{c_{n+1}c_{n}-c_{n}^2+c_{n}}
\nonumber\\
c_{n}c_{n+1}-c_{n}^2+c_{n} \ &\ \geq\ &\ c_{n+1}c_{n}+c_{n}-c_{n+1}c_{n-1}+c_{n+1}-c_{n-1}+1
\nonumber\\
c_{n}^2 \ &\ \leq\ &\ c_{n+1}c_{n-1}-c_{n+1}+c_{n-1}-1
\nonumber\\
c_{n+1}c_{n-1}+t^2(-r)^{n-1} \ &\ \leq\ &\ c_{n+1}c_{n-1}-c_{n+1}+c_{n-1}-1
\nonumber\\
0 \ &\ \leq\ &\ -c_{n+1}+c_{n-1}-t^2(-r)^{n-1}-1,
\end{eqnarray}
which contradicts Equation \eqref{II54}. Thus 
\begin{equation}\label{II45}
\frac{1}{c_{n}-c_{n-1}+1} \ <\ \frac{1}{c_{n}} + \frac{1}{c_{n+1}-c_{n}+1}. 
\end{equation}
Continuing the pattern, we obtain
\begin{equation}\label{II46}
\frac{1}{c_{n}-c_{n-1}+1} \ <\  \sum_{k = n}^{\infty} \frac{1}{c_{k}}
\end{equation}
since by a similar proof to Lemma \ref{lem1}, $\frac{1}{c_{n}-c_{n-1}+1}$ converges to $0$ as $n$ approaches infinity. Combining the two inequalities, we have
\begin{equation}\label{II47}
\frac{1}{c_{n}-c_{n-1}+1}  \ <\  \sum_{k = n}^{\infty} \frac{1}{c_{k}}\ <\ \frac{1}{c_{n}-c_{n-1}}.
\end{equation}
Therefore
\begin{equation}\label{II48}
\left\lfloor\left(\sum_{k = n}^{\infty} \frac{1}{c_{k}}\right)^{-1}\right\rfloor \ = \ 
\begin{cases}
c_{n}-c_{n-1}-1 & \text{{\rm if $n$ is odd and $n$}} \geq 1\\
c_{n}-c_{n-1}  & \text{{\rm if $n$ is even and $n$}} \geq 2,\\
\end{cases}
\end{equation}
which completes the proof for \emph{Case 2}.\\
\indent For \emph{Case 1}, note that since $t^2(-r)^{n-1} > 0$ when $r<0$, we simply need to repeat the steps we made for the cases when $n$ is odd in \emph{Case 2} since it has been already covered. Thus, when $q \geq 3$ and $r \geq -1$, for all $n$ we have
\begin{equation}\nonumber
\left\lfloor\left(\sum_{k = n}^{\infty} \frac{1}{c_{k}}\right)^{-1}\right\rfloor \ = \ c_n-c_{n-1}-1.
\end{equation}

\end{proof}
\begin{cor}\label{cor: II1}
%
% FIX - Once we fix Theorem 1.7 this must get changed.
% The result will not depend on the parity of $n$ since we have t^2(-r)^(n-1) >= 0 for all n. Should just have the first result c_{n}-c_{n-1}-1
%
For cobalancers corresponding to $(a,b)$ cobalancing numbers with $a \in \{1,2\}$, for any positive integer $n$, we have
\begin{equation}\label{II50}
\left\lfloor\left(\sum_{k = n}^{\infty} \frac{1}{c_{k}}\right)^{-1}\right\rfloor \ = \ 
c_{n}-c_{n-1}-1 
\end{equation}
where $c_{n}$ is the $n$\textsuperscript{{\rm th}} cobalancer, $c_{1} = 1$, and $c_{2} = \left(\frac{4b}{a}+2 \right)$. 
\end{cor}
\begin{proof}
By Theorem \ref{gen-cobalancer}, cobalancers corresponding to $(a,b)$ cobalancing numbers with $a \in \{1,2\}$ satisfy recurrence relations of the form
\begin{equation}\label{II51}
c_{n+1} \ = \ (2m+2)c_{n}-c_{n-1},
\end{equation}
where $m = 2b/a$, which corresponds with \emph{Case 1} of Theorem \ref{thm:II1}. We know from Theorem \ref{gen-cobalancer} that $c_{1}=1$, thus for all $n$, we have
\begin{eqnarray}\label{II52}
2\ & \leq& \ (c_{n+1}-c_{n})+(c_{n}-c_{n-1})\nonumber\\
1\ & \leq& \ c_{n+1}-c_{n-1}-1.
\end{eqnarray}
It follows that $c_2 = 2m+2$, and although the initial conditions are slightly different from those described in Theorem \ref{thm:II0}, the sequences are simply shifted so we can extend them to include $c_0 = 0$. Then these recurrences satisfy the conditions in \emph{Case 1} of Theorem \ref{thm:II1}, and the result directly follows.
\end{proof}

\begin{rek} \label{cor: II2}
Theorem \ref{thm:II1} also covers several other interesting depth two recurrence sequences such as the Fibonacci sequence, which was proven using a different method by Ohtsuka and Nakamura \cite{ON}, and the Pell sequence.
\end{rek}

%%%%%%%%%%%%%%%%%%%%%%%%%%%%%%%%%%%%%%%%%%%%%%%%%%%%%%%%%%
%%%%%%%%%%%%%%%%%%%%%%%%%%%%%%%%%%%%%%%%%%%%%%%%%%%%%%%%%%
%%%%%%%%%%%%%%%%%%%%%%%%%%%%%%%%%%%%%%%%%%%%%%%%%%%%%%%%%%
\subsection{Reciprocal Sums Related to Tribonacci Numbers}

%%%%%%%%%%%%%%%%%%%%%%%%%%%%%%%%%%%%%%%%%%%%%%%%%%%%%%%%%%
%%%%%%%%%%%%%%%%%%%%%%%%%%%%%%%%%%%%%%%%%%%%%%%%%%%%%%%%%%
\subsubsection{Reciprocal Sums of Every $n$\textsuperscript{{\rm th}} Tribonacci Numbers}\hspace*{\fill} \\

\begin{lem}\label{binet}
    According to \cite{Ce}, 
the Binet formula for Tribonacci numbers is 
\begin{equation}
    T_{n} \ = \ A\alpha ^{n}+B\beta^{n}+C\gamma^{n},
\end{equation} 
where $\alpha, \beta, \gamma$ are the three roots of the equation $x^{3}-x^{2}-x-1=0$, and 
\begin{equation}
\begin{split}
    A&=\frac{1}{-\alpha^{2}+4\alpha-1}\\
B&=\frac{1}{-\beta^{2}+4\beta-1}\\
C&=\frac{1}{-\gamma^{2}+4\gamma-1}.
\end{split}
\end{equation}
\end{lem}

Our starting point is the following result, proved by Komatsu in Section 2, Lemma 1 of \cite{Ko}.

\begin{lem}\label{c4alpha} For any positive integer $n$,
\begin{equation}\label{DI1}
    |T_n - c_4\cdot \alpha^n| \ < \ a\cdot d^n,
\end{equation}
where $c_4 = 0.33622811699, a = 0.51998, d = 0.7373527$ and $$\alpha\ =\ \frac{\sqrt[3]{19 + 3\sqrt{33}} + \sqrt[3]{19-3\sqrt{33}} + 1}{3} \ \approx\ 1.839286755.$$
\end{lem}

This result allows us to compute the closest integer to the reciprocal sum of every $m$\textsuperscript{th} Tribonacci number, assuming our initial index is sufficiently large.\\
\\
\textbf{Theorem \ref{sum of every nth}}
\textit{
Let $m$ be a positive integer. For large enough $n$, we have that
\begin{equation}\nonumber
    \left\{\left(\sum^{\infty}_{k = 0} \frac{1}{T_{n + mk}}\right)^{-1}\right\}\ =\ T_{n} - T_{n-m},
\end{equation}
where $\{j\}$ denotes the closest integer to $j$.
}
\begin{proof}
We first show that 
\begin{equation}\label{DI3}
    \left\{\left( \sum_{k \geq n} \frac{1}{T_{mk}}\right)^{-1}\right\} \ =\ T_{mn} - T_{mn-m}
\end{equation}
for large enough $n$. Note that this identity can be applied to prove the identity \ref{thm:DI2}.

From lemma $\ref{c4alpha},$ we have that 
\begin{eqnarray}\label{DI4}
    \frac{1}{T_{mk}} \ &=&\  \frac{1}{c_4\alpha^{mk} + O(d^{mk})} \nonumber\\
    \ &=&\  \frac{1}{c_4\alpha^{mk}\left(1 + O\left(\left(\frac{d}{\alpha}\right)^{mk}\right)\right)}\nonumber\\
    \ &=&\  \frac{1}{c_4\alpha^{mk}}\left(1 + O\left(\left(\frac{d}{\alpha}\right)^{mk}\right)\right)\nonumber\\
    \ &=&\  \frac{1}{c_4\alpha^{mk}} + O\left(\left(\frac{d}{\alpha^2}\right)^{mk}\right).
\end{eqnarray}

It follows that
\begin{eqnarray}\label{DI5}
    \sum_{k \geq n} \frac{1}{T_{mk}} \ &=&\  \frac{1}{c_4}\sum_{k\geq n}\frac{1}{\alpha^{mk}} + O\left(\sum_{k\geq n}\left(\frac{d}{\alpha^2}\right)^{mk}\right)\nonumber\\
    \ &=&\  \frac{\alpha^{m}}{c_4\alpha^{mn}(\alpha^m-1)} + O\left(\left(\frac{d}{\alpha^2}\right)^{mn}\right).
\end{eqnarray}

Therefore,
\begin{eqnarray}\label{DI6}
     \left( \sum_{k \geq n} \frac{1}{T_{mk}}\right)^{-1}\ & \ = \ &\ \left(\frac{\alpha^{m}}{c_4\alpha^{mn}(\alpha^m-1)}\left(1 + O\left(\left(\frac{d}{\alpha}\right)^{mn}\right)\right)\right)^{-1}\nonumber\\ &=&\ \frac{c_4\alpha^{mn}(\alpha^m-1)}{\alpha^{m}}\left(1 + O\left(\left(\frac{d}{\alpha}\right)^{mn}\right)\right)\nonumber\\
    &=&\  c_4\alpha^{mn} - c_4\alpha^{mn-m} + O(d^{mn})\nonumber\\
    &=&\  T_{mn} - T_{mn-m} + O(d^{mn}).
\end{eqnarray}
Once $n$ is sufficiently large, the error term $O(d^{mn})$ is less than $1/2$, leaving us the desired formula. Another proof that we will present in a moment will show that the formula also holds for smaller $n$. 

Note that in the proof of Theorem \ref{sum of every nth}, the starting index can be changed to any non-negative integer, and the proof is still valid. Thus, this concludes the first proof of the theorem. 
\end{proof}

We could also prove the formula without using big-O notation by working directly with complex numbers. \\
\\
\textbf{Alternate Proof of Theorem \ref{sum of every nth}}
We wish to show that 
\begin{equation}
    \left\{\left(\sum^{\infty}_{k = 0} \frac{1}{T_{n + mk}}\right)^{-1}\right\} \ =\ T_{n} - T_{n-m},
\end{equation}
in another way.
Let $p_n$ be
\begin{equation}\label{CI6}
    p_{n}\ :=\
    \frac{B}{A}\left(\frac{\beta}{\alpha}\right)^{n}+\frac{C}{A}\left(\frac{\gamma}{\alpha}\right)^{n}.
\end{equation}

We can therefore express the reciprocals of the Tribonacci numbers as
\begin{eqnarray}\label{CI7}
    \frac{1}{T_{n}} &\ =\ & \frac{1}{A\alpha ^{n}+B\beta^{n}+C\gamma^{n}}\nonumber\\
    &\ =\ & \frac{1}{A\alpha^{n}}\cdot \frac{1}{1+\frac{B}{A}\left(\frac{\beta}{\alpha}\right)^{n}+\frac{C}{A}\left(\frac{\gamma}{\alpha}\right)^{n}}\nonumber\\
    &\ =\ & \frac{1}{A\alpha^{n}}\cdot\frac{1}{1+p_n}.
\end{eqnarray}
We expand the above, using the infinite geometric series formula with ratio $-p_n$, and obtain
\begin{equation}\label{CI9}
    \frac{1}{T_{n}}\ =\  \frac{1}{A\alpha^{n}}\cdot \left(1-p_{n}+p_{n}^{2}-p_{n}^{3}+p_{n}^{4}+\cdots\right).
\end{equation}

Let $\ell_{n}=p_{n}-p_{n}^{2}+p_{n}^{3}-p_{n}^{4}+\cdots$, then, since $|p_{n}| \ <
\ 2\left| \frac{B}{A}\left(\frac{\beta}{\alpha}\right)^{n} \right| \ = \ 2 \frac{|B|\cdot |\beta|^{n}}{A\alpha^{n}} \ < \ 0.16$,
\begin{equation}\label{DI8}
    |\ell_{n}| \ = \ |p_n| \left |\frac{1}{1-p_n}\right | \ < \ 1.2|p_{n}|\ \leq\ 2.4 \frac{|B|\cdot |\beta|^{n}}{A\alpha^{n}}
\end{equation}
for all $n \geq 3$. \\
Thus,
\begin{equation}\label{DI9}
\begin{split}
    \left(\sum^{\infty}_{k=0}\frac{1}{T_{n+mk}}\right)^{-1} \ &=\ \left(\sum^{\infty}_{k=0}\frac{1}{A\alpha^{n+mk}}-\sum^{\infty}_{k=0}\frac{\ell_{n+mk}}{A\alpha^{n+mk}}\right)^{-1}\\
    \ &=\ \left(\frac{1}{A\alpha^{n}-A\alpha^{n-m}}- q_{n}\right)^{-1}\text{ , where }|q_{n}|\leq \frac{2.4|B||\beta|^{n}}{A^{2}\alpha^{2n}}\cdot \frac{\alpha^{2m}}{\alpha^{2m}-|\beta|^{m}}\\
    \ &=\ \frac{A\alpha^{n}-A\alpha^{n-m}}{1-\left(A\alpha^{n}-A\alpha^{n-m}\right)q_{n}}\\
    \ &=\ A\alpha^{n}-A\alpha^{n-m}+ \left(A\alpha^{n}-A\alpha^{n-m}\right)\left(r_{n}+r_{n}^{2}+r_{n}^{3}+r_{n}^{4}+\cdots\right),
\end{split}
\end{equation}
where $r_{n} = \left(A\alpha^{n}-A\alpha^{n-m}\right)q_{n}$. Since
\begin{equation}\label{DI10}
|r_{n}|\ = \  \left(A\alpha^{n}-A\alpha^{n-m}\right)|q_{n}|\ \leq \ A\alpha^{n}|q_{n}| \ \leq\ \frac{2.4|B||\beta|^{n}}{A\alpha^{n}}\cdot \frac{\alpha^{2m}}{\alpha^{2m}-|\beta|^{m}}\ \leq\ 0.025,
\end{equation}
for all $n\ \geq\ 5$, we have that
\begin{equation}\label{DI11}
    \left(\sum^{\infty}_{k=0}\frac{1}{T_{n+mk}}\right)^{-1} \ = \ A\alpha^{n}-A\alpha^{n-m}+
    s_{n},
\end{equation}
where 
\begin{equation}\label{DI12}
|s_{n}|\ \leq \ A\alpha^{n}\left|r_{n}+r_{n}^{2}+r_{n}^{3}+r_{n}^{4}+\cdots\right|\ \leq \ 1.03A\alpha^{n}|r_{n}| \ \leq\ 2.5 |B||\beta|^{n}\cdot \frac{\alpha^{2m}}{\alpha^{2m}-|\beta|^{m}}.
\end{equation}
Thus
\begin{eqnarray}\label{DI13}
    \left|T_{n}-T_{n-m}-\left(\sum^{\infty}_{k=0}\frac{1}{T_{n+mk}}\right)^{-1}\right| \ &\ = &\ 
    |B\beta^{n}-B\beta^{n-m}+C\gamma^{n}-C\gamma^{n-m}-s_{n}|\nonumber\\
    & \ \leq&\  |B\beta^{n-m}-C\gamma^{n-m}|+|B\beta^{n}-C\gamma^{n}|+|s_{n}|\nonumber\\
    &\ \leq&\  |B\beta^{n-m}-C\gamma^{n-m}|+|B\beta^{n}-C\gamma^{n}|\nonumber\\
    &&\ +\ 2.5|B||\beta|^{n}\cdot \frac{\alpha^{2m}}{\alpha^{2m}-|\beta|^{m}}.
\end{eqnarray}

Denote $f\left(x\right)=\left | B\beta^{x}-C\gamma^{x}\right | \left(x\in \mathbb{Z}\right)$. It is clear that $f\left(x\right)\in \mathbb{R}$. We will separate the values of $x$ into 2 cases. When $x$ is 0, 
\begin{eqnarray}\label{DI14}
    f(0)\ &\approx&\ 0.3966482802.
\end{eqnarray}
When $n\geq1$,
\begin{equation}\label{DI15}
    |f\left(x\right)| \ = \ 2|\Re\left(B\beta^{x}\right)|\ \leq\ 2|B||\beta|^{x}\leq2|B||\beta|^{2}\ \approx\ 0.3834086631.
\end{equation}
Thus we have that 
\begin{equation}
    \max |f(x)| \ = \ |f\left(0\right)| \ \approx\   0.3966482802. 
\end{equation} 
and hence for $x\geq 0$,
\begin{equation}
    |B\beta^{x}-C\gamma^{x}| \ \leq \ |B-C|.
\end{equation}
Therefore, we have that
\begin{equation}\label{DI16}
\begin{split}
    \left|T_{n}-T_{n-m}-\left(\sum^{\infty}_{k=0}\frac{1}{T_{n+mk}}\right)^{-1}\right| &\leq\
    |B\beta^{n-m}-C\gamma^{n-m}|+|B\beta^{n}-C\gamma^{n}|\\
    &\ \ \ \ \ +2.5|B||\beta|^{n}\cdot \frac{\alpha^{2m}}{\alpha^{2m}-|\beta|^{m}}\\
    &\ \leq\  |B-C|+2|B||\beta|^{n}+2.5|B||\beta|^{n}\cdot \frac{\alpha^{2}}{\alpha^{2}-|\beta|}\\
    &\ \leq\  |B-C|+2|B||\beta|^{9}\\
    &\ \ \ \ \ +2.5|B||\beta|^{9}\cdot \frac{\alpha^{2}}{\alpha^{2}-|\beta|}\text{ for all }n\geq9\\
    &\approx\  0.4836979971\\
    &<\  0.5.
\end{split}
\end{equation}
Thus, for all $n\geq9$ we have that
\begin{equation}\label{DI17}
\left \{ \left(\sum^{\infty}_{k=0}\frac{1}{T_{n+mk}}\right)^{-1}
\right \} \ = \  T_{n}-T_{n-m} \text{ }\left(n\geq m\right).
\end{equation}
The cases when $n<9$ can be checked by brute force. There are $36$ cases in total, and only the case $(n,m)=(1,1)$ does not satisfy the result above.
\hfill \qedsymbol

\begin{thm}\label{thm:DI31} For $n \ge k$,
\begin{eqnarray}
    &&\left \lfloor \left(\sum^{\infty}_{p=0}\frac{1}{T_{n+kp}}\right)^{-1}
    \right \rfloor \nonumber\\&=&\
\begin{cases}
T_{n}-T_{n-k}  & \text{ {\rm if} }
\sum^{\infty}_{p=0}
\frac{T_{n+kp}^{2}-T_{n+kp+k}T_{n+kp-k}}{T_{n+kp}\left(T_{n+kp}-T_{n+kp-k}\right)\left(T_{n+kp+k}-T_{n+kp}\right)}<0;\nonumber\\
T_{n}-T_{n-k}-1  & \text{ {\rm if} }
\sum^{\infty}_{p=0}
\frac{T_{n+kp}^{2}-T_{n+kp+k}T_{n+kp-k}}{T_{n+kp}\left(T_{n+kp}-T_{n+kp-k}\right)\left(T_{n+kp+k}-T_{n+kp}\right)}>0.\nonumber\\
\end{cases} \\ \
\end{eqnarray}
\end{thm}

\begin{proof} We have
\begin{equation}\label{DI32}
\begin{split}
    \frac{1}{T_{n}}-\frac{1}{T_{n}-T_{n-k}}+\frac{1}{T_{n+k}-T_{n}}
    \ &=\  \frac{1}{T_{n+k}-T_{n}}-\frac{T_{n-k}}{T_{n}\left(T_{n}-T_{n-k}\right)}\nonumber\\
    \ &=\  \frac{T_{n}^{2}-T_{n+k}T_{n-k}}{T_{n}\left(T_{n}-T_{n-k}\right)\left(T_{n+k}-T_{n}\right)},
\end{split}
\end{equation}
\begin{equation}\label{DI33}
    \frac{1}{T_{n}-T_{n-k}}\ =\ \frac{1}{T_{n}}+\frac{1}{T_{n+k}-T_{n}}-
\frac{T_{n}^{2}-T_{n+k}T_{n-k}}{T_{n}\left(T_{n}-T_{n-k}\right)\left(T_{n+k}-T_{n}\right)}.
\end{equation}

Modifying \eqref{DI33} for $\frac{1}{T_{n+k}-T_{n}}$ and so on we find
\begin{equation}\label{DI34}
    \frac{1}{T_{n}-T_{n-k}} \ = \ \sum^{\infty}_{p=0}\frac{1}{T_{n+kp}}-\sum^{\infty}_{p=0}
    \frac{T_{n+kp}^{2}-T_{n+kp+k}T_{n+kp-k}}{T_{n+kp}\left(T_{n+kp}-T_{n+kp-k}\right)\left(T_{n+kp+k}-T_{n+kp}\right)}.
\end{equation}
Therefore, the relationship between the reciprocal sum and $T_{n}-T_{n-k}$ depends on whether the last term of \eqref{DI34} is positive or negative. Since Theorem \ref{sum of every nth} provides the nearest integer of the reciprocal sum, we can now determine its floor.
\end{proof}

We can also prove a result for the alternating reciprocal sum.

\begin{thm}
For very large $n$, we have that
\begin{equation}\label{thm:DI35}
    \left\{\left(\sum^{\infty}_{k = n} \frac{(-1)^{k}}{T_{km-j}}\right)^{-1}\right\} \ =\  (-1)^n(T_{mn-j} + T_{mn-j-m})
\end{equation}
where $j \in \mathbb{Z}$ and $0\leq j < m.$
\end{thm}

\begin{proof}
Again, it suffices to show that
\begin{equation}\label{DI36}
\left\{\left(\sum_{k\geq n}\frac{(-1)^{k}}{T_{km}}\right)^{-1}\right\} \ =\  (-1)^n(T_{mn} + T_{mn-m}).
\end{equation}
Using \eqref{DI4},
\begin{equation}\label{DI37}
    \begin{split}
        \sum_{k\geq n}\frac{(-1)^{k}}{T_{km}} \ &=\   \frac{1}{c_4}\sum_{k\geq n}\left(-\frac{1}{\alpha^{m}} \right)^{k}+ O\left(\sum_{k\geq n}\left(-\left(\frac{d}{\alpha^{2}}\right)^{m}\right)^{k}\right)\\
        \ &=\  \frac{\alpha^{m}}{c_4(-\alpha)^{mn}(\alpha^m +1)}+ O\left(\left(-\left(\frac{d}{\alpha^2}\right)^{m}\right)^{n}\right).
    \end{split}
\end{equation}
It follows that
\begin{equation}\label{DI38}
    \begin{split}
         \left(\sum_{k\geq n}\frac{(-1)^{k}}{T_{km}}\right)^{-1}\ &=\   \left(\frac{\alpha^{m}}{c_4(-\alpha)^{mn}(\alpha^m +1)}\left(1+ O\left(\left(\frac{d}{\alpha}\right)^{mn}\right)\right)\right)^{-1}\nonumber\\
         \ &=\  \frac{c_4(-\alpha)^{mn}(\alpha^m +1)}{\alpha^{m}}\left(1+ O\left(\left(\frac{d}{\alpha}\right)^{mn}\right)\right)\\
         \ &=\  (-1)^{mn}(c_4\alpha^{mn}+c_4\alpha^{mn-m}) + O(-d^{mn})\\
         \ &=\ (-1)^n(T_{mn} + T_{mn-m}) + O(d^{mn}).
    \end{split}
\end{equation}
Again, once $n$ is large enough, the $O(d^{mn})$ term is less than $1/2$, as desired.
\end{proof}

%%%%%%%%%%%%%%%%%%%%%%%%%%%%%%%%%%%%%%%%%%%%%%%%%%%%%%%%%%%%%%%%%%%%%%%%%%%%%%%%%%%%%%%%%
%%%%%%%%%%%%%%%%%%%%%%%%%%%%%%%%%%%%%%%%%%%%%%%%%%%%%%%%%%%%%%%%%%%%%%%%%%%%%%%%%%%%%%%%%

\subsubsection{Reciprocal Sums of Sums of Tribonacci Numbers}\hspace*{\fill} \\

According to [AK,
Lemma 1.$(ii)$],
\begin{equation}\label{EI1}
    \sum_{k = 1}^n T_k \ = \ \frac{T_{n+2} + T_n -1}{2}.
\end{equation}

We can use their result to derive the following identity.

\begin{thm}
When $n$ is sufficiently large, we have that
\begin{equation}\label{thm:EI2}
    \left\{\left(\sum_{k = n}^{\infty} \frac{1}{\sum_{i = 1}^{mk}T_i})\right)^{-1}\right\}\ =\  \frac{T_{mn+2} + T_{mn} - T_{mn-m+2} -T_{mn-m}}{2}.
\end{equation}
\end{thm}

\begin{proof}
According to equation \eqref{EI1},
\begin{eqnarray}\label{EI3}
        \sum_{k = n}^{\infty} \frac{1}{\sum_{i = 1}^{mk}T_i} \ &=&\  \sum_{k = n}^{\infty} \frac{2}{T_{mk+2} + T_{mk} - 1}\nonumber\\
        &=&\  \sum_{k = n}^{\infty} \frac{2}{c_4\alpha^{mk+2} + O(d^{mk+2}) + c_4\alpha^{mk} + O(d^{mk})-1}\nonumber\\
        &=&\  \sum_{k = n}^{\infty}\frac{2}{c_4\alpha^{mk}(1+\alpha^2) + O(d^{mk})}\nonumber\\
        &=&\  \sum_{k = n}^{\infty}\frac{2}{c_4\alpha^{mk}(1+\alpha^2)\left(1 + O\left(\left(\frac{d}{\alpha}\right)^{mk}\right)\right)}\nonumber\\
        &=&\  \left(\frac{2}{c_4}\right)\sum_{k = n}^{\infty}\frac{1}{\alpha^{mk}(1+\alpha^2) + O\left(\left(\frac{d}{\alpha^2}\right)^{mk}\right)}.
\end{eqnarray}

Next, note that
\begin{eqnarray}\label{EI4}
        &&\ \left(\frac{2}{c_4}\right)\sum_{k = n}^{\infty}\frac{1}{\alpha^{mk}(1+\alpha^2) + O\left(\left(\frac{d}{\alpha^2}\right)^{mk}\right)}\nonumber\\
        &=&\  \left(\frac{2}{c_4}\right)\left(\frac{\alpha^m}{\alpha^{mn}(1+\alpha^2)(\alpha^m-1)} + O\left(\left(\frac{d}{\alpha^2}\right)^{mn}\right) \right).
\end{eqnarray}
Therefore,
\begin{equation}\label{EI5}
    \begin{split}
        \left(\sum_{k = n}^{\infty} \frac{1}{\sum_{i = 1}^{mk}T_i}\right)^{-1} \ &=\  \left(\left(\frac{2}{c_4}\right)\left(\frac{\alpha^m}{\alpha^{mn}(1+\alpha^2)(\alpha^m-1)}\left(1+ O\left(\left(\frac{d}{\alpha}\right)^{mn}\right) \right)\right)\right)^{-1}\\
        \ &=\  \frac{1}{2}\left(\frac{c_4\alpha^{mn}(1+\alpha^2)(\alpha^m-1)}{\alpha^m} \left(1+ O\left(\left(\frac{d}{\alpha}\right)^{mn}\right)  \right) \right)\\
        \ &=\  \frac{1}{2}\left(c_4(\alpha^{mn+2} +\alpha^{mn} - \alpha^{mn-m+2} + \alpha^{mn-m}) + O(d^{mn}) \right)\\
        \ &=\  \frac{T_{mn+2} + T_{mn} - T_{mn-m+2} -T_{mn-m}}{2} + O(d^{mn})
    \end{split}
\end{equation}
and the $O(d^{mn})$ term is less than $1/2$ when $n$ is sufficiently large, completing the proof.
\end{proof}

We immediately obtain the following.

\begin{cor}
For large enough $n$ and any integer $j < mn$, we have that
\begin{equation}\label{thm:EI6}
    \left\{\left(\sum_{k = n}^{\infty} \frac{1}{\sum_{i = 1}^{mk-j}T_i})\right)^{-1}\right\}\ = \ \frac{T_{mn-j+2} + T_{mn-j} - T_{mn-m-j+2} -T_{mn-j-m}}{2}.
\end{equation}
\end{cor}

%%%%%%%%%%%%%%%%%%%%%%%%%%%%%%%%%%%%%%%%%%%%%%%%%%%%%%%%%%%%%%%%%%%%%%%%%%%%%%%%%%%%%%%%%
%%%%%%%%%%%%%%%%%%%%%%%%%%%%%%%%%%%%%%%%%%%%%%%%%%%%%%%%%%%%%%%%%%%%%%%%%%%%%%%%%%%%%%%%%
\subsubsection{Reciprocal Sums of Generalized Tribonacci Numbers}\hspace*{\fill} \\

In this section, we discuss sums of generalized Tribonacci numbers. However, due to convergence issues, we only discuss generalizations to some constant coefficient recurrences of depth three. If we were to generalize the result to all generalized Tribonacci numbers, the proof would be much more difficult. 

\begin{defi} Given integers $p, q, r, X, Y, Z$, let $\{G_n\}$ be the sequence with initial terms
\begin{equation}\label{FI1}
    G_{0} \ = \ p,\ \ \ G_{1} \ = \ q,\ \ \ G_{2}\ = \ r 
\end{equation} and satisfying the recurrence
\begin{equation}\label{FI2}
    G_{n}\ =\  XG_{n-1}+YG_{n-2}+ZG_{n-3} 
\end{equation}
for all $n \geq 3$.
\end{defi}

Let $\alpha, \beta, \gamma$ be the three roots to the equation
\begin{equation}\label{FI3}
    s^{3}-Xs^{2}-Ys-Z\ =\ 0
\end{equation}
where $|\alpha| \ge |\beta| \ge |\gamma|$, the General Binet formula for $G_{n}$ says that  
\begin{equation}\label{FI4}
    G_{n}\ =\  A\alpha^{n}+B\beta^{n}+C\gamma^{n}
\end{equation} if there are three distinct roots to the corresponding characteristic polynomial associated to the recurrence relation.\footnote{If there are repeated roots there are trivial modifications, multiplying the exponential terms by a polynomial of degree one less than the multiplicity.}
\begin{rek}
If the first $3$ numbers are the same as the standard Tribonacci sequence ($T_0 = 0, T_1 = 0, T_2 = 1$), using $G_{0},G_{1},G_{2}$ to solve for the coefficients, we get
\vspace{-3mm}
\begin{equation}\label{FI5}
    \begin{split}
        A\ &=\  \frac{1-\beta-\gamma}{(\alpha-\gamma)(\alpha-\beta)}\\
        B\ &=\  \frac{1-\gamma-\alpha}{(\beta-\gamma)(\beta-\alpha)}\\
        C\ &=\  \frac{1-\beta-\alpha}{(\gamma-\beta)(\gamma-\alpha)}.
    \end{split}
\end{equation}
\end{rek}
\begin{thm}
If 
\begin{enumerate}
    \item $\alpha$ is a triple root (that is, $\alpha=\beta=\gamma$),
    \item $\alpha =\beta $ are two real roots and $|\gamma|<1$, and
    \item $\alpha$ is a real root and $|\beta|, |\gamma| < 1$,
\end{enumerate}
then for $n \ge k$
\begin{equation}\label{thm:FI6}
\left(\sum^{\infty}_{p=0}\frac{1}{G_{n+kp}}\right)^{-1}\ =\ G_{n}-G_{n-k}.
\end{equation}
\end{thm}

\begin{proof} 
For case 1, we have
    \begin{equation}\label{FI7}
        \begin{split}
            \left ( \sum^{\infty}_{p=0}\frac{1}{G_{n+kp}}\right)^{-1}\ &=\
            \left ( \sum^{\infty}_{p=0}\frac{1}{(A+B+C)\alpha^{n+kp}}\right)^{-1}\\
            \ &=\ \left (\frac{1}{(A+B+C)(\alpha^{n}-\alpha^{n-k})}\right)^{-1}\\
            \ &=\  G_{n}-G_{n-k}.
        \end{split}
    \end{equation}
    
For case 2, we have 
    \begin{equation}\label{FI9}
        \begin{split}
            \frac{1}{G_{n}}\ &=\  \frac{1}{(A+B)\alpha^{n}+C\gamma^{n}}\\
            \ &=\ \frac{1}{(A+B)\alpha^{n}}\cdot \frac{1}{1+\frac{C}{A+B}\left( \frac{\gamma}{\alpha}\right)^{n}}\\
            \ &=\  \frac{1}{(A+B)\alpha^{n}}\left(1+O\left(\frac{\gamma^{n}}{\alpha^{n}}\right)\right)\\
            \ &=\  \frac{1}{(A+B)\alpha^{n}}+O\left(\frac{\gamma^{n}}{\alpha^{2n}}\right).
        \end{split}
    \end{equation}
    Thus,
    \begin{equation}\label{FI10}
        \begin{split}
            \left(\sum^{\infty}_{p=0}\frac{1}{G_{n+kp}}\right)^{-1}\ &=\
            \left(\frac{1}{(A+B)(\alpha^{n}-\alpha^{n-k})}\right)^{-1}+O\left(\frac{\gamma^{n}}{\alpha^{2n}}\right)\\
            \ &=\  \left(\frac{1}{(A+B)(\alpha^{n}-\alpha^{n-k})}\left(1+O\left(\frac{\gamma^{n}}{\alpha^{n}}\right)\right)\right)^{-1}\\
            \ &=\  (A+B)(\alpha^{n}-\alpha^{n-k})\left(1+O\left(\frac{\gamma^{n}}{\alpha^{n}}\right)\right)\\
            \ &=\  A(\alpha^{n}-\alpha^{n-k})+B(\alpha^{n}-\alpha^{n-k})+O(\gamma^{n})\\
            \ &=\  A(\alpha^{n}-\alpha^{n-k})+B(\alpha^{n}-\alpha^{n-k})+C(\gamma^{n}-\gamma^{n-k})+O(\gamma^{n})\\
            \ &=\ G_{n}-G_{n-k}+O(\gamma^{n}),
        \end{split}
    \end{equation}
completing the proof.\\\\
For case 3, the proof is very similar to Theorem \ref{sum of every nth}.
\end{proof}
%%%%%%%%%%%%%%%%%%%%%%%%%%%%%%%%%%%%%%%%%%%%%%%%%%%%%%%%%%%%%%%%%%%%%%%%%%%%%%%%%%%%%%%%%%%%%%%%%%%%%%%%%%%%%%%%%%%%%%%%%%%%%%%%%%%%%%%%%%%%%%%%%%%%%%%%%%%%%%%%%%%%%%%%%%%%%%%%%%%%%%%%%%%%%%%%%%%%%%%%%%%%%%%%%%%%%%%%%%%%%%%%%%%%%%%%%%%%%%%%%%%%%%%%%%%%%%%%%%%%%%%%%%%%%%%%%%%%%%%%%%%%%%%%%%%%%%%%%%%%%%%%%%%%%%%%%%%%%%%%%%%%%%%%%%%%%%%%%%%%%%%%%%%%%%%%%%%%%%%%%%%%%%%%%%%%%%%%%%%%%%%%%%%%%%%%%%%%%%%%

%%%%%%%%%%%%%%%%%%%%%%%%%%%%%%%%%%%%%%%%%%%%%%%%%%%%%%%%%%%%%%%%%%%%%%%%%%%%%%%%%%%%%%%%%%%%%%%%%%%%%%%%%%%%%%%%%%%%%%%%
%%%%%%%%%%%%%%%%%%%%%%%%%%%%%%%%%%%%%%%%%%%%%%%%%%%%%%%%%%%%%%%%%%%%%%%%%%%%%%%%%%%%%%%%%%%%%%%%%%%%%%%%%%%%%%%%%%%%%%%%
%%%%%%%%%%%%%%%%%%%%%%%%%%%%%%%%%%%%%%%%%%%%%%%%%%%%%%%%%%%%%%%%%%%%%%%%%%%%%%%%%%%%%%%%%%%%%%%%%%%%%%%%%%%%%%%%%%%%%%%%
\section{Coefficients Such That Every Integer or No Integers are Balancing Numbers}

\noindent\textbf{Theorem \ref{thm:JI1}}
\textit{
If $a$ and $b$ are relatively prime integers, then the set of coefficients $(a, b) = (3, 1)$ are the only $(a,b)$ coefficients for which every integer $n$ is a balancing number.}
\begin{proof}
By the definition for balancing numbers with coefficients $(a, b)$, we have
\begin{eqnarray} \label{JI1}
   a \left( \frac{(n-1)n}{2}\right) & = & b\left(\frac{(n+r)(n+r+1)}{2}-\frac{(n+1)n}{2}\right)\\
   \frac{a}{b}(n-1)n & = & (n+r)(n+r+1)-(n+1)n \nonumber \\
    r^2+(2n+1)r+\left(-\frac{a}{b}n^2+\frac{a}{b}n\right) & = & 0\nonumber \\
    r & = & \frac{-(2n+1)\pm \sqrt{(4+4\frac{a}{b})n^2+(4-4\frac{a}{b})n+1}}{2}. \nonumber
\end{eqnarray}
Since $r$ must be an integer, let
\begin{equation} \label{JI2}
\sqrt{\left(4+4\frac{a}{b}\right)n^2+\left(4-4\frac{a}{b}\right)n+1} \ = \ m,
\end{equation}
where $m$ is an integer for all $n$. Then
\begin{equation}
\left(4+4\frac{a}{b}\right)n^2+\left(4-4\frac{a}{b}\right)n+1 \ = \ m^2. \nonumber
\end{equation}
Let $m=x n+1$, where $x \in \mathbb{R}$. Then
\begin{eqnarray}
\left(4+4\frac{a}{b}\right)n^2+\left(4-4\frac{a}{b}\right)n+1& \  =& \ (x n+1)^2 \nonumber \\
& \ =& \ x^2n^2+2x n+1. \nonumber
\end{eqnarray}
By matching corresponding terms we get
\begin{equation}
    x^2 \ = \ 4+4\frac{a}{b} \nonumber
\end{equation}
and
\begin{equation}
    2x \ = \ 4-4\frac{a}{b}. \nonumber
\end{equation}
Adding both equations gives
\begin{equation}
    x^2+2x-8 \ = \ 0 \nonumber
\end{equation}
so $x=-4$ and $x=2$. 

If $x=-4$, then
\begin{eqnarray} \label{JI3}
    (-4)^2 & \ = & \ 4+4\frac{a}{b} \nonumber \\
    3 & \ = & \ \frac{a}{b}.
\end{eqnarray}
Since $a$ and $b$ must be relatively prime, $a=3$ and $b=1$.

If $x=2$, then
\begin{eqnarray} \label{JI4}
    (2)^2 & \ = & \ 4+4\frac{a}{b} \nonumber\\
    0 & \ = & \ 4\frac{a}{b}. 
\end{eqnarray}
The values $a$ and $b$ must be greater than zero, so this does not give any additional solutions. Therefore, $(3,1)$ is the only set of coefficients where $a$ and $b$ are relatively prime such that every integer $n$ is a balancing number. Plugging in $a = 3$ and $b = 1$ into Equation \eqref{JI1} gives $r = n-1$, so for coefficients $(3,1)$ the balancer will always be one less than the corresponding balancing number.
\end{proof}

\noindent\textbf{Theorem \ref{thm:JI2}}
\textit{
Coefficients $(a,b)$ do not exist such that every integer $n$ is a cobalancing number.}
\begin{proof}
By the definition for cobalancing numbers with coefficients $(a, b)$ and Equation \eqref{eqn1}, we have
\begin{eqnarray} \label{JI5}
 r & \ = & \ \frac{-(2n+1)+ \sqrt{(4+4\frac{a}{b})n^2+(4+4\frac{a}{b})n+1}}{2}. 
\end{eqnarray}
Since $r$ must be an integer, let
\begin{equation} \label{JI6}
\sqrt{\left(4+4\frac{a}{b}\right)n^2+\left(4+4\frac{a}{b}\right)n+1} \ = \ m,
\end{equation}
where $m$ is an integer for all $n$. Let $m=x n+1$, where $x \in \mathbb{R}$. Then
\begin{eqnarray}
\left(4+4\frac{a}{b}\right)n^2+\left(4+4\frac{a}{b}\right)n+1& \ = & \ (x n+1)^2 \nonumber \\
& \ = & \ x^2n^2+2x n+1. \nonumber
\end{eqnarray}
By matching corresponding terms we get
\begin{equation}
 x^2 \ = \ 4+4\frac{a}{b} \nonumber   
\end{equation}
and 
\begin{equation}
    2x \ = \ 4+4\frac{a}{b}. \nonumber
\end{equation}
We find that $x=0$ or $x=2$. If $x=2$, we have
\begin{eqnarray} \label{JI7}
    (2)^2 & \ = & \ 4+4\frac{a}{b} \nonumber\\
    0 & \ = & \ 4\frac{a}{b}.
 \end{eqnarray}

As $a$ and $b$ are nonzero, this does not yield any solutions for $a$ and $b$. If $x=0$, we have
 \begin{eqnarray} \label{JI16}
    0 & \ = & \ 4+4\frac{a}{b}. \nonumber \\
    \frac{a}{b}& \  = & \ -1.
 \end{eqnarray}
Substituting this result into the last equation of \eqref{JI5}, we get
\begin{eqnarray} \label{JI17}
    r & \ = & \frac{-(2n+1)\pm \sqrt{0n^{2}+0n+1}}{2} \nonumber \\
    r & \ = & -n\text{ or }-n-1,
\end{eqnarray}
which is in contradiction with the range of $r$. Therefore there are no sets of coefficients $(a,b)$ such that every integer $n$ is a cobalancing number.
\end{proof}

\noindent\textbf{Theorem \ref{thm:JI4}}
\textit{
For all coefficients $(a,b)$ such that $a=16y^2+16y+3$ and $b=1$, where $y$ is a positive integer, the only cobalancing number is $n=y$ with corresponding cobalancer $r=4y^2+3y$.}
\begin{proof}
Let $a=x^2-1$, where $x=4y+2$ and $y$ is a positive integer. Then by Equation \eqref{JI5},
\begin{equation}\label{JI10}
    r \ = \ \frac{-(2n+1) + \sqrt{4x^2n^2+4x^2n+1}}{2}
\end{equation}
since $r \in \mathbb{Z}^+$, we have $\sqrt{4x^2n^2+4x^2n+1}=m$, where $m$ is a positive integer.
 
By definition, we have 
\begin{equation} \label{JI11}
    x \ = \ 4y+2,
\end{equation}
where $y$ is an integer. Then
\begin{eqnarray}
    a   &\ = &\ (4y+2)^2-1
    \nonumber\\
    a   &\ = &\ 16y^2+16y+3. \nonumber
\end{eqnarray}
We can also substitute into our equation for $r$.
\begin{eqnarray}
     r & \ = & \ \frac{-(2n+1) + \sqrt{4x^2n^2+4x^2n+1}}{2} \nonumber \\
     r & \ = & \ \frac{-(2n+1) + \sqrt{4(4y+2)^2n^2+4(4y+2)^2n+1}}{2}. \nonumber
\end{eqnarray}
Again, since $r$ is a positive integer,
\begin{equation}\label{JI12}
     \sqrt{4(4y+2)^2n^2+4(4y+2)^2n+1} \ = \ m,
\end{equation}
where $m$ is an integer. 

Suppose $y=n$. Then Equation \eqref{JI12} simplifies to
\begin{eqnarray} \label{JI13}
     \sqrt{64n^4+128n^3-80n^2+16n+1} & \ = & \ m \nonumber \\
     \sqrt{(8n^2+8n+1)^2} & \ = & \ m \nonumber \\
     8n^2+8n+1 & \ = & \ m. 
\end{eqnarray}
The above equation shows that all values of $n$, where $n \geq 1$, provide a solution for $r$ when $y=n$. The value of $n$ uniquely determines $r$, so for each $(n, r)$ there can only be one possible set of corresponding coefficients $(a,b)$, where $a$ and $b$ are relatively prime. As $y$ determines $a$ and $b=1$, we have found all the sets of coefficients. Since we have found solutions for all $n$, this is an exhaustive list; note $y=0$ cannot yield any additional solutions. 

Since this is an exhaustive list for all $n$ values, we can set $y=n$ and find
\begin{eqnarray} \label{JI14}
     r &\ = & \ \frac{-(2n+1) + \sqrt{(8n^2+8n+1)^2}}{2}\nonumber\\ &\ = & \ \frac{-(2n+1)+(8n^2+8n+1)^2}{2}\nonumber\\ &\ = & \ 4n^2+3n.
\end{eqnarray}
Since for each value of $y$ (which uniquely determines the coefficients), there is only one solution $n=y$ (other than the case $y=0$), then for each set of coefficients under these conditions, there will only be one cobalancing number. 

Therefore, for coefficients $a=x^2-1$ and $b=1$, when $x=4y+2$ and $y=1, 2, 3, \ldots$, the only cobalancing number will be $n=y$ with a corresponding cobalancer $r=4y^2+3y$.
\end{proof}

\begin{conj} \label{con:JI1}
For coefficients $a=x^2-1$, where $x$ is congruent to $0, 1,$ or $3 \mod 4$, and $b=1$, the results from the code in Appendix \ref{code} suggests that
\begin{equation} \label{JI15}
    \sqrt{4x^2n^2+4x^2n=1} \ = \ m
\end{equation}
has no integer solutions $(x,n,m)$ where $x>1$ and $n\geq1$, so the Diophantine equations corresponding with these cases have no solutions within the specified bounds. 
Note that the proof of Theorem \ref{thm:JI4} covers the case for x congruent to $2 \mod 4$. 
% \ 
\subparagraph{Case 1:} $x\equiv 0 \mod 4$: \\
The Diophantine equation $4(4y)^2n^2+4(4y)^2n+1=m^2$ has no solutions where $n\geq1$. 
% \ \\
\subparagraph{Case 2:} $x\equiv 1 \mod 4$: \\
The Diophantine equation $4(4y+1)^2n^2+4(4y+1)^2n+1=m^2$ has no solutions where $n\geq1$.

% \ \\
\subparagraph{Case 4:} $x\equiv 3 \mod 4$: \\
The Diophantine equation $4(4y+3)^2n^2+4(4y+3)^2n+1=m^2$ has no solutions where $n\geq1$.
\end{conj}
%%%%%%%%%%%%%%%%%%%%%%%%%%%%%%%%%%%%%%%%%%%%%%%%%%%%%%%%%%%%%%%%%%%%%%%%%%%%%%%%%%%%%%%%%%%%%%%%%%%%%%%%%%%%%%%%%%%%%%
\section{(a,b) Square Balancing and Cobalancing Numbers}
\begin{defi} \label{def:KI1}
The $(a, b)$ square balancing numbers are positive integers $n$ such that the equality
\begin{equation}\label{KI1}
    a\left(1^2+2^2+\cdots+(n-1)^2\right)\ = \ b\left( (n+1)^2+(n+2)^2+\cdots+(n+r)^2\right)
\end{equation}
is satisfied for some non-negative integer $r$, and $r$ is called the square balancer of $n$. 
\end{defi}

Note that our definition of square balancing numbers differs from Panda's definition of higher-order balancing numbers in \cite{Pan}, since we chose to include the trivial solution $(n,r)=(1,0)$.

\begin{figure}[ht]
    \begin{center}
    \scalebox{.6}{\includegraphics{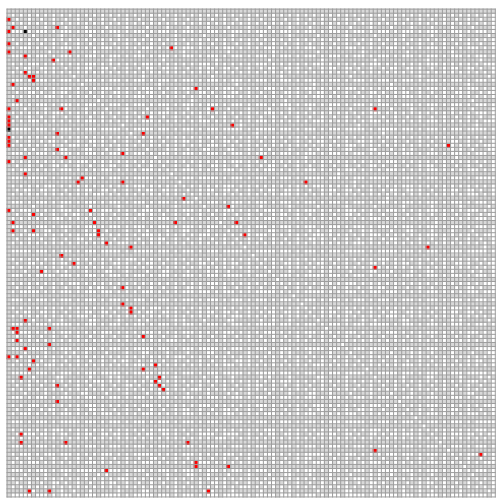}}
    \caption{\label{fig:balancingsquares} This $120\times120$ grid shows the frequency of square balancing numbers for given coefficients, starting with $(a,b)=(1,1)$ in the top left corner. Values of $a$ increase downwards and values of $b$ increase rightward. Coefficients that are not relatively prime are skipped, and are marked in white. Light gray indicates only one square balancing number, which must be the trivial solution $(n,r)=(1,0)$. Red indicates that there exists a second solution, and black indicates that a third solution exists. We did not find any instances with greater than three balancing numbers.}
    \end{center}
\end{figure}

Panda conjectured in \cite{Pan} that there does not exist any $(1,1)$ square balancing numbers other than the trivial solution $(n,r)=(1,0)$. The results of our code in Appendix \ref{code} support this conjecture, and suggest that for each $(a,b)$, there exists no more than three $(a,b)$ square balancing numbers (this includes the trivial solution). Based on Figure \ref{fig:balancingsquares}, however, we did not find any correlation between coefficients $(a,b)$ and the existence of solutions $(n,r)$.

\begin{defi}\label{def:KI2}
The $(a, b)$ square cobalancing numbers are positive integers $n$ such that the equality
\begin{equation}\label{KI2}
    a\left(1^2+2^2+\cdots+n^2\right)\ = \ b\left( (n+1)^2+(n+2)^2+\cdots+(n+r)^2\right)
\end{equation}
is satisfied for some non-negative integer $r$, and $r$ is called the square cobalancer of $n$. 
\end{defi}

\begin{figure}[ht]
    \begin{center}
    \scalebox{.6}{\includegraphics{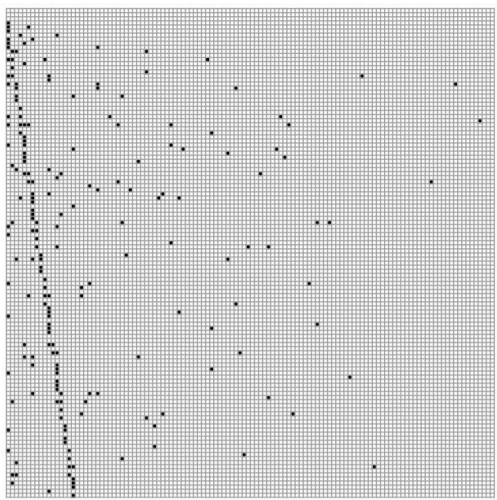}}
    \caption{\label{fig:cobalancingsquares} This $120\times120$ grid shows the frequency of square cobalancing numbers for given coefficients, starting with $(a,b)=(1,1)$ in the top left corner. Values of $a$ increase downwards and values of $b$ increase rightward. Coefficients that are not relatively prime are skipped so that the pattern is more noticeable. White indicates that no square cobalancing numbers were found, and black indicates that one solution exists. We did not find any instances with greater than one square cobalancing number.}
    \end{center}
\end{figure}

The results of our code in Appendix \ref{code} suggest that for each $(a,b)$, there is no more than one $(a,b)$ square cobalancing number. Based on Figure \ref{fig:cobalancingsquares}, certain values of $(a,b)$ containing a solution follow a distinct pattern, while the remaining values seem to have no correlation with the existence of solutions. The pattern is periodic and symmetric, repeating itself after every $42$ values for $a$ and $6$ values for $b$. For each set of coefficients $(a,b)$ with a corresponding solution $(n,r)$ contained in the pattern, $\frac{a-b}{r}=m$, where $m$ is an integer, and either $n=r$ or $n=r-1$.

%%%%%%%%%%%%%%%%%%%%%%%%%%%%%%%%%%%%%%%%%%%%%%%%%%%%%%%%%%%%%%%%%%%%%%%%%%%%%%%%%%%%%%%%%%%%%%%%%%%%%%%%%%%%%%%%%%%%%%%%
%%%%%%%%%%%%%%%%%%%%%%%%%%%%%%%%%%%%%%%%%%%%%%%%%%%%%%%%%%%%%%%%%%%%%%%%%%%%%%%%%%%%%%%%%%%%%%%%%%%%%%%%%%%%%%%%%%%%%%%%
%%%%%%%%%%%%%%%%%%%%%%%%%%%%%%%%%%%%%%%%%%%%%%%%%%%%%%%%%%%%%%%%%%%%%%%%%%%%%%%%%%%%%%%%%%%%%%%%%%%%%%%%%%%%%%%%%%%%%%%%

\section{Future Work}
\begin{itemize}
    \item We found that many $(a,b)$ balancing and cobalancing numbers and their balancers and cobalancers can be expressed as depth five recurrences of the form $(1,K,-K,-1,1)$. Is it possible to express all $(a,b)$ balancing and cobalancing numbers in this form, and is there a generalized formula for the recurrence for any coefficients $(a,b)$?  \\ \ 
    
    \item We proved Theorem \ref{thm:II0} for $q \geq 2$, but we suspect that the Theorem should hold for any nonzero real number $q$. In particular, we would need to prove that Lemma \ref{IIlemma} holds for all $q \in \mathbb{R}_{\ne 0}$. \\ \
    
    \item Can we find a generalized formula for $A_{n}^2$ for a recurrence of the form $A_{n+1} = q A_{n} + r A_{n-1} + s$? This would help further generalize the formulas for reciprocal sums.\\ \ 
    
    \item When do coefficients $(a,b)$ have no balancing or cobalancing numbers? We found cases for both, but they are likely not exhaustive.\\ \ 

    \item Finally, there is much to explore regarding square $(a,b)$ balancing and cobalancing numbers. Do coefficients $(a,b)$ exist with greater than three square balancing numbers or greater than one square cobalancing number? Can we prove that the pattern we found for square cobalancing numbers continues for all $(a,b)$? Does a similar pattern exist for square balancing numbers?
\end{itemize}

\section{Acknowledgements}
We thank Joyce Qu and the referee for their careful reading and helpful comments.

\appendix
%%%%%%%%%%%%%%%%%%%%%%%%%%%%%%%%%%%%%%%%%%%%%%%%%%%%%%%%%%%%%%%%%%%%%%%%%%%%%%%%%%%%%%%%%%%%%%%%%%%%%%%%%%%%%%%%%%%%%%%%
%%%%%%%%%%%%%%%%%%%%%%%%%%%%%%%%%%%%%%%%%%%%%%%%%%%%%%%%%%%%%%%%%%%%%%%%%%%%%%%%%%%%%%%%%%%%%%%%%%%%%%%%%%%%%%%%%%%%%%%%
%%%%%%%%%%%%%%%%%%%%%%%%%%%%%%%%%%%%%%%%%%%%%%%%%%%%%%%%%%%%%%%%%%%%%%%%%%%%%%%%%%%%%%%%%%%%%%%%%%%%%%%%%%%%%%%%%%%%%%%%
\section{(a,b) Balancing Numbers Interesting Case}

\begin{thm}\label{thm:JI3}
There do not exist any $(8,1)$ balancing numbers.
\end{thm}
\begin{proof}
Plugging $a=8$ and $b=1$ into Equation \eqref{JI1} gives
\begin{equation} \label{JI8}
    r \ = \ \frac{-2n-1\pm \sqrt{36n^2-28n+1}}{2}.
\end{equation}
For $r$ to be an integer, the equation $36n^2-28n+1=y^2$ (where $y$ is also an integer) must have an integer solution for $n$. Solving the Diophantine equation yields the following four solutions in the form of $(n,y)$: $(0,1), (1,-3), (0,-1), (1,3).$ Since $n\neq0$, the only solution for $n$ is $n=1$. Evaluating Equation \eqref{JI8} with $n=1$ gives
\begin{eqnarray} \label{JI9}
    r & \ = & \ \frac{-3\pm\sqrt{9}}{2} \nonumber \\
    r & \ = & \ -3, 0.
\end{eqnarray}
However, since $r > 0$, there do not exist any balancing numbers with coefficients $(8,1)$.
\end{proof}

%%%%%%%%%%%%%%%%%%%%%%%%%%%%%%%%%%%%%%%%%%%%%%%%%%%%%%%%%%%%%%%%%%%%%%%%%%%%%%%%%%%%%%%%%%%%%%%%%%%%%%%%%%%%%%%%%%%%%%%%
%%%%%%%%%%%%%%%%%%%%%%%%%%%%%%%%%%%%%%%%%%%%%%%%%%%%%%%%%%%%%%%%%%%%%%%%%%%%%%%%%%%%%%%%%%%%%%%%%%%%%%%%%%%%%%%%%%%%%%%%
%%%%%%%%%%%%%%%%%%%%%%%%%%%%%%%%%%%%%%%%%%%%%%%%%%%%%%%%%%%%%%%%%%%%%%%%%%%%%%%%%%%%%%%%%%%%%%%%%%%%%%%%%%%%%%%%%%%%%%%%

\section{Mathematica Code} \label{code}
Throughout this paper, we use results from code written in Mathematica, available at
\begin{itemize}
    \item \bburl{https://bit.ly/3JvaocG} \\ \

    \item \bburl{https://bit.ly/3LONIWu}.
\end{itemize}

The program contains code for finding recurrences describing sequences of $(a,b)$ balancing and cobalancing numbers. There is also code to find $(a,b)$ square balancing and cobalancing numbers which we used to create Figures \ref{fig:balancingsquares} and \ref{fig:cobalancingsquares}.

%%%%%%%%%%%%%%%%%%%%%%%%%%%%%%%%%%%%%%%%%%%%%%%%%%%%%%%%%%%%%%%%%%%%%%%%%%%%%%%%%%%%%%%%%%%%%%%%%%%%%%%%%%%%%%%%%%%%%%%%
%%%%%%%%%%%%%%%%%%%%%%%%%%%%%%%%%%%%%%%%%%%%%%%%%%%%%%%%%%%%%%%%%%%%%%%%%%%%%%%%%%%%%%%%%%%%%%%%%%%%%%%%%%%%%%%%%%%%%%%%
%%%%%%%%%%%%%%%%%%%%%%%%%%%%%%%%%%%%%%%%%%%%%%%%%%%%%%%%%%%%%%%%%%%%%%%%%%%%%%%%%%%%%%%%%%%%%%%%%%%%%%%%%%%%%%%%%%%%%%%%

\ \\

\end{document}